\documentclass[3p,12pt]{elsarticle}

\usepackage{latexsym,amssymb,amsmath,amsthm,amsfonts}
\usepackage{graphicx}
\usepackage{epstopdf}
\usepackage{caption}

\theoremstyle{plain}
\newtheorem{theorem}{Theorem}
\newtheorem{lemma}{Lemma}
\newtheorem{assumption}{Assumption}
\newtheorem{definition}{Definition}
\theoremstyle{remark}
\newtheorem{rem}{Remark}
\newtheorem{example}{Example}
\renewcommand*{\today}{December 23, 2014}

\begin{document}

\begin{frontmatter}

\title{\textbf{On the rate of convergence to Rosenblatt-type distribution}\\[5mm]
\small \textit{Short title: Rate  of convergence to Rosenblatt-type distribution}}

\author[qu]{Vo Anh}
\ead{v.anh@qut.edu.au}

\author[ca]{Nikolai Leonenko}
\ead{LeonenkoN@cardiff.ac.uk}

\author[la]{Andriy Olenko\corref{cor1}}
\ead{a.olenko@latrobe.edu.au}

\cortext[cor1]{Corresponding author. Phone: +61-3-9479-2609 \quad  Fax:  +61-3-9479-2466}
\address[qu]{School of Mathematical Sciences, Queensland University of Technology,\\ Brisbane, Queensland, 4001, Australia}
\address[ca]{School of Mathematics, Cardiff University,
Senghennydd Road,\\ Cardiff CF24 4AG, United Kingdom}
\address[la]{Department of Mathematics and Statistics, La Trobe University,\\ Victoria 3086, Australia}

Will appear in the Journal of Mathematical Analysis and Applications. The final publication is available at http://dx.doi.org/10.1016/j.jmaa.2014.12.016

\begin{abstract}
 The main result of the article is the rate of convergence to the Rosenblatt-type distributions in non-central limit theorems. Specifications of the main theorem  are discussed for several scenarios. In particular, special attention is paid to the Cauchy, generalized Linnik's, and local-global distinguisher random processes and fields. 
 Direct analytical methods are used to investigate the rate of convergence in the uniform metric.
\end{abstract}


\end{frontmatter}

\section{Introduction}

This paper studies local functionals of homogeneous random fields with long-range dependence, which appear in various applications in signal processing, geophysics, telecommunications, hydrology, etc. The reader can find more details about long-range dependent processes and fields in  \cite{dou1, gne,iv,leo1,wac} and the references therein. In particular, \cite{dou1} discusses different definitions of long-range dependence in terms of the autocorrelation function (the integral of the correlation function diverges) or the spectrum (the spectral density has a singularity at zero). The case when the summands/integrands are functionals of a long-range dependent Gaussian process is of great importance in the theory of limit theorems for sums/integrals of dependent random variables.  It was shown by Taqqu \cite{ta0,ta2} and Dobrushin and Major \cite{dob} that, comparing with the central limit theorem, long-range dependent summands can produce different normalizing coefficients and non-Gaussian limits.  The volumes \cite{dou1} and \cite{pec} give excellent surveys of the field. For multidimensional results of this type see \cite{iv,leo0,leo1}. Some most recent results can be found  in \cite{iv1,mink,ole2}. 

Despite recent progress in the non-central limit theory  there has been remarkably little fundamental theoretical study on rates of convergence in non-central limit theorems. The rate of convergence to the Gaussian distribution for a local functional of Gaussian random fields with long-range dependence was first obtained in \cite{leo0}. This result was applied to investigate the convergence of random solutions of the multidimensional Burgers equation in~\cite{leo2}. 
The only publications, which are known to the authors, on the rate of convergence to non-Gaussian distributions in the non-central limit theorem are \cite{bre,anh}. These publications investigate particular cases of stochastic processes. The Hermite power variations of a discrete-time fractional Brownian motion were studied in \cite{bre}. The article  \cite{anh} investigated the specific one-dimensional case of the Cauchy stochastic process and some facts used in the paper require corrections. To the best of our knowledge, the rate of convergence has never been studied in the  general  context of non-central limit theorems for non-Gaussian limit distributions. This work was intended as an attempt to obtain first results in this direction.

Our focus in this paper is on fine convergence properties of functionals of long-range dependent Gaussian fields. The paper establishes the rate of convergence in limit theorems for random fields, which is also new for the case of stochastic processes. 
It also generalizes the result of \cite{anh}, which was obtained for  a stochastic process with a fixed Cauchy 
covariance function, to integral functionals of random fields over arbitrary convex sets.  In addition, the paper corrects some proofs in \cite{anh}.  Specific important examples of the Cauchy, generalized Linnik, and local-global distinguisher random processes and fields, which have been recently used to separate a fractal dimension and the Hurst effect \cite{gne}, are considered.

To estimate distances between distributions in the limit theorems for non-linear transformations of Gaussian stochastic processes 
Nourdin and Peccati proposed an approach based on the Malliavin calculus and Stein's method, see \cite{nor,nor1} and the references therein.  The cases of the standard normal distribution and the centred Gamma distribution were considered and the limit theorems for the weakly dependent case were obtained. In \cite{bre} the Malliavin calculus and Stein's method were applied to obtain error bounds for Hermite power variations of a fractional Brownian motion. Central and non-central limit theorems for the Hermite variations of the anisotropic fractional Brownian sheet and the distance between a normal law and another law were studied in \cite{rev} and extended to the multidimensional case in  \cite{bre1}. However, to the best of our knowledge, there are no extensions of these results to  more general classes of covariance functions in the multidimensional case considered in this paper. In contrast we use direct analytical  probability methods to investigate the rate of convergence in the uniform (Kolmogorov) metric of long-range dependent random fields to the Rosenblatt-type distributions.  

The class of Rosenblatt-type distributions is contained in the wide class of  non-Gaussian Hermite distributions, which  can be defined by its representation in the form of multiple Wiener-It\^{o} stochastic integrals with respect to the complex Gaussian white noise random measure. The Rosenblatt distribution is a specific element from this class, which has been widely used recently in the probability theory and also appeared in a statistical context as the asymptotic distribution of certain estimators. There are power series expressions for the characteristic functions of the Rosenblatt distribution.  For a comprehensive exposition of the Rosenblatt distribution and process we refer the reader to \cite{gar,leotau,ta0,ta1,ta3,tud}.
The approach presented in the present paper seems to be suitable even in more general situations of the Hermite limit distributions.

The results were obtained under assumptions similar to the standard ones in \cite{rob} and the references therein. Rather general assumptions were chosen to  describe various asymptotic scenarios for correlation and spectral functions. Some simple sufficient conditions and examples of correlation models satisfying the assumptions are discussed in Sections~\ref{sec4} and~\ref{sec6}.  

As a bonus, some other new results of independent
interest in the paper are: the boundedness of probability densities of the Rosenblatt-type distributions, asymptotics at the origin of the spectral densities of the Cauchy, generalized Linnik, and local-global distinguisher random processes and fields, and the representation of the spectral density of the local-global distinguisher random processes.

The article is organized as follows.  In Section~\ref{sec1} we recall some basic definitions and formulae of the spectral theory of random fields. Section~\ref{sec2} introduces the key assumptions and auxiliary results. The main result is presented in Section~\ref{sec3} and its specifications to various important cases are demonstrated in Section~\ref{sec4}. Discussions and short conclusions are presented in Section~\ref{sec6}.

Some computations in Examples~\ref{ex5} and \ref{ex6} were performed by using Maple 15.0 of Waterloo Maple Inc. and verified by Mathematica  9.0 of Wolfram Research, Inc.

\section{Notations}\label{sec1}

In what follows $\left| \cdot \right| $ and $\left\| \cdot \right\| $ denote the Lebesgue measure and the Euclidean distance in~$\mathbb{R}^{d},$ respectively. We use the symbols $C$ and $\delta$ to denote constants which are not important for our exposition. Moreover, the same symbol  may be used for different constants appearing in the same proof.

We consider a measurable mean square continuous zero-mean homogeneous  isotropic real-valued
random field  $\eta(x),\ x\in \mathbb{R}^{d},$ defined on a probability space $(\Omega,\mathcal{F},P),$ 
with the covariance function
\[
\textrm{B}(r):=\mathbf{Cov}\left( \eta(x),\eta(y)\right) =\int_{0}^{\infty }Y_{d}(rz)\,\mathrm{d}\Phi(z),\ x,y\in \mathbb{R}^{d},
\]
where $r:= \left\| x-y\right\|,$ $\Phi(\cdot)$ is the isotropic spectral measure,  the function $Y_{d}(\cdot)$ is defined by
\[Y_{d}(z):=2^{(d-2)/2}\Gamma \left(\frac{d}{2}\right)\ J_{(d-2)/2}(z)\ z^{(2-d)/2},\quad
z\geq 0,
\]
$J_{(d-2)/2}(\cdot)$ is the Bessel function of the first kind of order $(d-2)/2.$

\begin{definition}  The random field $\eta (x),$ $x\in \mathbb{R}^{d},$ as defined above is said to possess an absolutely continuous spectrum if   there exists a function $f(\cdot)$ such that
$$
\Phi(z)=2\pi^{d/2}\Gamma^{-1}\left(d/2\right)\int_0^z u^{d-1}f(u)\,\mathrm{d}u,\quad z\ge 0,\quad u^{d-1}f(u)\in L_{1}(\mathbb{R}_{+}).
$$
 The function $f(\cdot)$ is
called the isotropic spectral density function of the field $\eta (x).$
\end{definition}
The field $\eta (x)$ with an absolutely continuous spectrum has the isonormal spectral representation
\[
\eta (x)=\int_{\mathbb{R}^d}e^{i(\lambda ,x)}\sqrt{f\left( \left\| \lambda \right\|
\right) }W(\mathrm{d}\lambda ),
\]
where $W(\cdot )$ is the complex Gaussian white noise random measure on $\mathbb{R}^d.$

Consider a Jordan-measurable convex bounded set $\Delta \subset \mathbb{R}^{d},$ such that
 $\left| \Delta \right| >0$ and $\Delta$ contains the origin in its interior. Let $\Delta (r),r>0,$ be the homothetic image of
the set $\Delta,$ with the centre of homothety at the origin and the coefficient $r>0,$
that is $\left| \Delta (r)\right| =r^{d}\left| \Delta \right|.$

Consider the uniform distribution on $\Delta (r)$
with the probability density function
(pdf)  $r^{-d}\left|\Delta \right|^{-1}  \chi_{\raisebox{-3pt}{\scriptsize $\Delta(r)$}}(x),$ $x\in\mathbb{R}^{d},$
where $\chi_{\raisebox{-3pt}{\scriptsize $A$}}(\cdot)$ is  the indicator function of  a set $A.$

\begin{definition} Let $U$ and $V$ be two random vectors which are independent and uniformly distributed inside the set $\Delta (r).$ We denote by $\psi _{\Delta (r)}(z
),$ $z \geq 0,$ the pdf of the distance $\left\| U-V\right\| $ between $U$ and $V.$
\end{definition}
  Note that $\psi _{\Delta (r)}(z
)=0$ if $z > diam\left\{ \Delta (r)
\right\}.$ Using the above notations, we obtain the  representation
\[
\int_{\Delta (r)}\int_{\Delta (r)}\varUpsilon(\left\| x-y\right\| )\,\mathrm{d}x
\,\mathrm{d}y=
\left| \Delta \right| ^{2}r^{2d}\mathbf{E}\ \varUpsilon(\left\| U-V\right\| )=
\]
\begin{equation}\label{dint}
=\left| \Delta \right| ^{2}r^{2d}\int_{0}^{diam\left\{ \Delta
(r)\right\} }\varUpsilon(z)\ \psi _{\Delta (r)}(z)\,\mathrm{d}z ,
\end{equation}
where $\varUpsilon(\cdot)$ is an integrable Borel function.

\begin{rem} If $\Delta (r)$ is the ball $v(r):=\{x\in \mathbb{R
}^{d}:\left\| x\right\| <r\},$ then

\[
\psi _{v(r)}(z)=d\,r^{-d} z^{d-1}I_{1-(z/2r)^{2}}\left( \frac{d+1}{2}
,\frac{1}{2}\right) ,\quad 0\leq z \leq 2r,  \]
where
\[
I_{\mu }(p,q):=\frac{\Gamma (p+q)}{\Gamma (p)\ \Gamma (q)}\int_{0}^{
\mu }u^{p-1}(1-u)^{q-1}\,\mathrm{d}u,\quad \mu \in ( 0,1],\quad p>0,\ q>0 ,
\]
is the incomplete beta function, see \cite{iv}.
\end{rem}

\begin{rem} Let $H_{k}(u)$, $k\geq 0$, $u\in \mathbb{R}$, be the Hermite polynomials, see \cite{pec}. If $(\xi _{1},\ldots ,\xi _{2p})$ is $2p$-dimensional
zero mean Gaussian vector with
\[
\mathbf{E}\xi _{j}\xi _{k} =
\begin{cases} 1, &\mbox{if } k=j; \\
r_{j}, & \mbox{if } k=j+p\ \mbox{and } 1\leq j\leq p,\\
0, &\mbox{otherwise,}
 \end{cases}
\]
then
\[
\mathbf{E}\ \prod_{j=1}^{p}H_{k_{j}}(\xi _{j})H_{m_{j}}(\xi
_{j+p})=\prod_{j=1}^{p}\delta _{k_{j}}^{m_{j}}\ k_{j}!\ r_{j}^{k_{j}}.
\]
\end{rem}

The Hermite polynomials form a complete orthogonal system
in the Hilbert space
\[
{L}_{2}(\mathbb{R},\phi (w )\, dw) =\left\{G:\int_{
\mathbb{R}}G^{2}(w) \phi ( w )\,\mathrm{d}w<\infty \right\}, \quad \phi
(w):=e^{-\frac{w^{2}}{2}}/\sqrt{2\pi }.
\]

An arbitrary function $G(w)\in {L}_{2}(\mathbb{R},\phi(w )\, dw)$ admits the mean-square convergent expansion
\begin{equation}\label{herm}
G(w)=\sum_{j=0}^{\infty }\frac{C_{j}H_{j}(w)
}{j !}, \qquad C_{j }:=\int_{\mathbb{R}}G(w)H_{j
}(w)\phi ( w )\,\mathrm{d}w.  \end{equation}

By Parseval's identity
\begin{equation}\label{par}
\sum_{j=0}^\infty\frac{C_{j}^{2}}{j !}
=\int_{\mathbb{R}}G^2(w) \phi ( w )\,\mathrm{d}w.
\end{equation}

\begin{definition} {\rm\cite{ta0}}  Let $G(w)\in {L}_{2}(\mathbb{R},\phi ( w)\,
dw)$ and assume there exists an integer $\kappa\geq 1$ such that $C_{j }=0$, for all $0\leq j\leq \kappa-1,$ but $C_{\kappa }\neq 0.$  Then $\kappa$ is called the Hermite rank of $G(\cdot)$ and denoted by
$H \mbox{rank}\,G.$
\end{definition}

\begin{definition} {\rm\cite{bin}} A measurable function $L:(0,\infty )\rightarrow (0,\infty )$ is
called slowly varying at infinity if for all $t >0,$%
\begin{equation*}
\lim\limits_{\lambda\to \infty }\frac{L(\lambda t)}{L(\lambda)}=1.
\end{equation*}
\end{definition}

 By the representation theorem~\cite[Theorem 1.3.1]{bin}, there exists $C > 0$ such that for all $r \ge C$ the function $L(\cdot)$ can be written in the form
\begin{equation}\label{L}
L(r) = \exp \left(\zeta_1(r) + \int_C^r \frac{\zeta_2(u)}{u}\,\mathrm{d}u \right), \end{equation}
where $\zeta_1(\cdot)$ and $\zeta_2(\cdot)$ are such measurable and bounded functions that $\zeta_2(r)\to 0$ and $\zeta_1(r)\to C_0$ $(|C_0|<\infty),$   when $r\to \infty.$

If $L(\cdot)$ varies slowly and $a>0$ then  $
r^a L(r)\to \infty,$ $r^{-a} L(r)\to 0,$ when $r\to \infty,$ see Proposition 1.3.6 \cite{bin}.

\section{Assumptions and auxiliary results}\label{sec2}

In this section we list the main assumptions and some auxiliary results from \cite{mink} which will be used to obtain the rate of convergence in non-central limit theorems. The detailed discussion of the main assumptions is given in Section~\ref{sec6}. We also prove the boundedness of the pdf of the Rosenblatt-type distributions.

\begin{assumption}\label{ass1} { Let  $\eta (x),$ $x\in \mathbb{R}^{d}$, be a homogeneous isotropic Gaussian random
field with $\mathbf{E}\eta (x)=0$ and the covariance function $B(x)$ such that

\[B(0)=1,\quad B(x)=\mathbf{E}\eta (0) \eta (x)= \left\Vert x\right\Vert
^{-\alpha }L(\left\Vert x\right\Vert ),
\]
where $L(\left\Vert \cdot\right\Vert )$ is a function slowly varying at infinity.}
\end{assumption}

In this paper we restrict our consideration  to $\alpha \in (0,d/\kappa),$  where $\kappa$ is the Hermite rank in Definition 3.  For such $\alpha$ the covariance function $B(x)$ satisfying Assumption~\ref{ass1} is not integrable, which corresponds to the case of long-range dependence. 

Let us denote
\[
K_r :=\int_{\Delta(r)}G(\eta (x))\,\mathrm{d}x \quad \mbox{and}
\quad K_{r,\kappa} :=\frac{C_\kappa}{\kappa !}
\int_{\Delta(r)}H_\kappa (\eta (x))\,\mathrm{d}x,
\]
where $C_\kappa$ is defined by (\ref{herm}).

\begin{theorem}\label{th4}{\rm\cite{mink}}
Suppose that $\eta (x),$ $x\in\mathbb{R}^d,$ satisfies Assumption~{\rm\ref{ass1}}  and    $H {\rm rank}\,G(\cdot)=\kappa\ge 1.$
If there exists the limit distribution for at least one of the
random variables
\[
\frac{K_r}{\sqrt{ \mathbf{Var} K_r}}\quad \mbox{and}\quad  \frac{K_{r,\kappa}}{\sqrt{ \mathbf{Var} \ K_{r,\kappa}}},\]
then the limit distribution of the other random variable exists too and the limit distributions coincide when $r\to \infty.$
\end{theorem}

\begin{rem} By the property
$\lim_{r\to \infty }{\mathbf{Var}\, K_{r}}/{\mathbf{Var}\, K_{r,\kappa}} = 1
$ (see \cite{mink})
Theorem~\ref{th4} holds if the first random variable is replaced by
${K_r}/{\sqrt{ \mathbf{Var} \ K_{r,\kappa}}}.$
\end{rem}

\begin{assumption}\label{ass2} The random field $\eta(x),$ $x \in \mathbb{R}^d,$ has the spectral density

\begin{equation}\label{f}
f(\left\| \lambda \right\| )=  c_2(d,\alpha )\left\| \lambda \right\| ^{\alpha
-d}L\left( \frac 1{\left\| \lambda \right\| }\right)+\varepsilon(\left\| \lambda \right\|),
\end{equation}
where
$\varepsilon(t)=t^{\alpha
-d}L\left( 1/t\right)\cdot{\cal O}\left(\min\left(t^\upsilon,1\right)\right),$ as $\max\left(t,1/t\right) \to +\infty,$
 \[
c_2(d,\alpha ):=\frac{\Gamma \left( \frac{d-\alpha }2\right) }{2^\alpha \pi ^{d/2}\Gamma
\left( \frac \alpha 2\right) },\]
and $L(\left\Vert \cdot\right\Vert )$ is a function which is locally bounded slowly varying at infinity and  satisfies for sufficiently large $r$ the condition
\begin{equation}\label{gr}\left|1-\frac{L(tr)}{L(r)}\right|\le
C\,t^{\nu}/r^{q},\ t\ge 1,
\end{equation}
where $\upsilon>0,$ $q>0,$ and $\nu$ are constants.
\end{assumption}

\begin{rem}
For $d=1$ Assumption~\ref{ass2} is similar to the conditions employed in \cite{rob} and the references therein to describe the asymptotic behaviour of the spectral density at zero. For example, the conditions~(\ref{f}) and (\ref{gr})  are the equivalents of Assumptions~3 and 4 in \cite{rob}.  Some sufficient conditions for Assumption~\ref{ass2} are discussed in Remark~\ref{sec6} and  \cite{rob}.
\end{rem}
\begin{rem}
For $d>1$ the situation is more complex.  Under some additional conditions (for example, monotonicity and essential positivity)  Assumptions~\ref{ass1} and~\ref{ass2} are  linked by Abelian and Tauberian theorems. However, they do not imply each other in general, consult  \cite{leoole,ole12}. Thus, to investigate the rate of convergence we need both assumptions.  Moreover, Assumption~\ref{ass2} provides more detailed information about the asymptotic behaviour of the spectral density at zero than one can obtain from the corresponding Tauberian theorem.
\end{rem}

The following Lemma shows that (\ref{gr}) can be replaced by a "stronger" condition.
\begin{lemma}~\label{lem0}
The condition~{\rm(\ref{gr})} is equivalent to \begin{equation}\label{gr0} \left|1-\frac{L(tr)}{L(r)}\right|\le
C/r^{q},\ t\ge 1,\ q>0.
\end{equation}
\end{lemma}
\begin{proof}
The condition~(\ref{gr}) implies that, for each $t\ge 1,$ ${L(tr)}/{L(r)}-1={\cal O}\left(r^{-q}\right),$ as $r\to\infty,$ where ${\cal O}\left(r^{-q}\right)$ may be different for different values of $t.$ Notice that $r^{-q}$ has positive decrease because its upper Matuszewska index (refer to Section~2.1.2~\cite{bin})  is $-q<0.$  Then, by the representation theorem for slowly varying functions with remainder, see Corollary 3.12.3~\cite{bin}, 
\begin{equation}\label{Cr}L(r)=C(1+C(r)), \ \mbox{as} \ r\to\infty,\end{equation} 
where $C(r)={\cal O}\left(r^{-q}\right).$ Therefore, by (\ref{Cr})
\[\sup_{t\ge 1}\,\left|1-\frac{L(tr)}{L(r)}\right|=\frac{\sup_{t\ge 1}\left|C\left(r\right)-C\left(tr\right)\right|}{1+C\left(r\right)}\le \frac{C}{r^q}+\sup_{t\ge 1} \frac{C}{(tr)^q}=
{\cal O}\left(r^{-q}\right),\ r\to\infty,
\]
and the condition~(\ref{gr}) can be replaced by (\ref{gr0}).
\end{proof}

\begin{rem}\label{rem5}
An example of a sufficient condition for (\ref{gr}) is that $L(\cdot)$ is differentiable and its derivative satisfies
\begin{equation}\label{logL}
|L'(r)|={\cal O}(L(r)/r^{1+q}),\quad r\to +\infty.
\end{equation}

Indeed,  by Theorem~1.5.3 \cite{bin},  there exist $r_0>0$ and $C>0$ such that for all $r\ge r_0$ it holds
\[\left|1-\frac{L(tr)}{L(r)}\right|\le \left|\frac{r(t-1)\sup_{u\in [r,rt]}L'(u)}{L(r)}\right|\le r(t-1)\sup_{u\in [r,rt]}\left|\frac{L'(u)}{L(u)}\right| \]
\[\times \sup_{u\in [r,rt]}\frac{u^\delta}{r^\delta}\cdot\frac{\sup_{u\in [r,rt]}u^{-\delta}L(u)}{r^{-\delta}L(r)}\le C\,\frac{t^{1+\delta}}{r^q},\quad t\ge 1.\]

Notice, that (\ref{logL}) can be rewritten as $|\ln'(L(r))|={\cal O}(r^{-1-q}),$ as $r\to +\infty.$ Therefore, if the function $\zeta_1(\cdot)$ is differentiable in (\ref{L})  we obtain the sufficient condition
\[\left|\zeta_1'(r) + \frac{\zeta_2(r)}{r}\right|={\cal O}(r^{-1-q}),\quad r\to +\infty.\]

A few simple examples of functions satisfying the condition~(\ref{gr}) for sufficiently large~$r$ are $L(r)=a_0,$  $L(r)=\left(a_0+a_1\,r^{-q}\right)^{a_2},$ and $L(r)=a_0\exp(a_1/r^q),$ where $a_0>0,$ $a_1$ and $a_2$ are constants. The  function $L(r)=\ln r$ does not satisfy the condition~(\ref{gr}).
\end{rem}

\begin{rem}\label{rem6}
Note that Assumption~\ref{ass1} implies $L(t)={\cal O}(t^\alpha),$ $t \to +0.$ Therefore, if Assumption~\ref{ass1} holds true then the condition $\varepsilon(t)=t^{\alpha
-d}L\left( 1/t\right)\cdot{\cal O}\left(\min\left(t^\upsilon,1\right)\right)$ is equivalent to $\varepsilon(t)={\cal O}\left(t^{
-d}\right),$ when $t \to +\infty.$ Hence, the condition (\ref{f}) is equivalent to $f(\left\| \lambda \right\|)={\cal O}(\left\| \lambda \right\|^{-d}),$ when $\left\| \lambda \right\| \to +\infty.$ Thus, if Assumption~\ref{ass1} is fulfilled then for the case $\left\| \lambda \right\| \to +\infty$ one can use $f(\left\| \lambda \right\|)={\cal O}(\left\| \lambda \right\|^{-d})$  instead of the condition (\ref{f}) in Assumption~{\rm\ref{ass2}}.

\end{rem}

Let us denote the Fourier transform of the indicator function of the  set $\Delta$ by
\begin{equation}\label{K}
K_{\Delta}(x):=\int_{\Delta }e^{i(x,u)} \,\mathrm{d}u,\quad
x\in\mathbb{R}^d.
\end{equation}

\begin{lemma}\label{finint} {\rm \cite{mink}}
If $\tau_1,...,\tau_\kappa,$ $\kappa\ge 1,$ are such positive constants that $\sum_{i=1}^\kappa \tau_i <d,$  then
\[
\int_{\mathbb{R}^{d\kappa}}|K_{\Delta}(\lambda _1+\cdots
+\lambda _\kappa)|^2 \frac{\mathrm{d}\lambda _1\ldots \,\mathrm{d}\lambda _\kappa}{\left\| \lambda
_1\right\| ^{d-\tau_1}\cdots \left\| \lambda _\kappa\right\| ^{d-\tau_\kappa}}<\infty .
\]
\end{lemma}

\begin{theorem}{\rm\cite{mink}}\label{th2} Let $\eta(x),$ $x\in \mathbb{R}^d,$ be a homogeneous isotropic Gaussian random
field with $\mathbf{E}\eta(x)=0.$   If Assumptions~{\rm \ref{ass1}} and {\rm \ref{ass2}} hold, then for $r\to \infty$ the finite-dimensional distributions of
\[X_{r,\kappa}:=r^{(\kappa\alpha )/2-d}L^{-\kappa/2}(r)\int_{\Delta(r)}H_\kappa(\eta (x))\,\mathrm{d}x\]
converge weakly to the finite-dimensional distributions of
\[
X_\kappa(\Delta) :=c_2^{\kappa/2}(d,\alpha )  \int_{\mathbb{R}^{d\kappa}}^{{\prime }}K_{\Delta }(\lambda _1+\cdots
+\lambda _\kappa) \frac{W(\mathrm{d}\lambda _1)\ldots W(\mathrm{d}\lambda _\kappa)}{\left\| \lambda
_1\right\| ^{(d-\alpha )/2}\cdots \left\| \lambda _\kappa\right\| ^{(d-\alpha )/2}},\]
where $\int_{\mathbb{R}^{d\kappa}}^{{\prime}}$ denotes the multiple Wiener-It\^{o} integral.
\end{theorem}

\begin{definition}
The probability distribution of $X_2(\Delta)$ will be called the Rosenblatt-type distribution. \end{definition}
 It is a generalization of the Rosenblatt distribution to arbitrary set~$\Delta.$ Consult \cite{gar,leotau,ta0,ta1,ta3,tud} on various properties and applications of the Rosenblatt distribution.

\begin{lemma}\label{lem3}
The Rosenblatt-type distribution has a bounded probability density function:
\[
\sup_{z\in\mathbb R}\,p_{X_2(\Delta)}\left( z\right) =\sup_{z\in\mathbb R}\,\frac{\mathrm{d}}{\mathrm{d}z}P\left( X_2(\Delta) \leq z\right)<+\infty.
\label{22}
\]
\end{lemma}
\begin{proof}
By Theorem 1  in \cite{dav} (also consult \cite{nor2} for some recent results)  it follows that the probability distribution of $X_{2}(\Delta)$ is absolutely continuous. 
To show that $X_{2}(\Delta)$ has a bounded density we will use Theorem 2  of \cite{dav} which is valid for general measurable spaces. In our case it requires the existence of linearly independent functions $h_1(\cdot),h_2(\cdot)\in L_2\left(\mathbb{R}^d\right)$ such that
\begin{equation}\label{g01}
\int_{\mathbb{R}^{2d}}\frac{K_{\Delta }(\lambda _1+\lambda _2)\,h_1(\lambda_1)h_1(\lambda_2)}{\left\| \lambda
_1\right\| ^{(d-\alpha )/2} \left\| \lambda _2\right\| ^{(d-\alpha )/2}}\,\mathrm{d}\lambda _1\mathrm{d}\lambda _2>0,
\end{equation}
and
\begin{equation}\label{g02}
\int_{\mathbb{R}^{4d}}\triangledown_K(\lambda _1,\lambda _2,\lambda _3,\lambda _4)\,h_1(\lambda_1)h_1(\lambda_2)h_2(\lambda_3)h_2(\lambda_4)\,\prod_{j=1}^4\frac{\mathrm{d}\lambda _j}{\left\| \lambda_j\right\| ^{(d-\alpha )/2}}>0,
\end{equation}
where $\triangledown_K(\lambda _1,\lambda _2,\lambda _3,\lambda _4):=K_{\Delta }(\lambda _1+\lambda _2){K_{\Delta }(\lambda _3+\lambda _4)}-K_{\Delta }(\lambda _1+\lambda _3)K_{\Delta }(\lambda _2+\lambda _4).$

Let us choose 
\begin{equation}\label{h12} h_j(\lambda,A_0):= \chi_{\raisebox{-3pt}{\scriptsize $v(A_0)$}}(\lambda) \left(\frac{\ell_j\{\Delta\}}{2\pi}\right)^{d/2} \frac{J_{{d}/{2}} (\left\| \lambda\right\|\cdot \ell_j\{\Delta\})}{
\left\| \lambda\right\|^{\alpha/2}}, 
\end{equation}
where $A_0$ is a positive number, $v(A_0)$ is the Euclidean ball of
radius $A_0,$ and
 $\ell_j\{\Delta\}>0,$ $j=1,2.$   For convenience, we will also use the definition (\ref{h12}) in the case  $A_0=+\infty$ assuming that $\chi_{\raisebox{-3pt}{\scriptsize $v(+\infty)$}}(\lambda)\equiv 1.$ 

We will need the following asymptotic properties of  
the Bessel function of the first kind, see (8.402) and (8.451) \cite{gra},
\[J_{d/2}(z)\sim \sqrt{\frac{2}{\pi
z}}\cos\left(z-\pi(d+1)/4\right),\
z\to\infty,\quad
J_{d/2}(z)\sim \frac{z^{d/2}}{2^{d/2}\Gamma\left(d/2+1\right)},\
z\to 0. \]

By the definition (\ref{h12}),
\begin{itemize}
\item $h_j(\cdot,A_0)\in L_2\left(\mathbb{R}^d\right),$ $j=1,2,$ are radial functions with compact supports for which 
$\tilde{h}_j(\lambda,A_0):=\left\| \lambda
\right\| ^{(\alpha-d)/2}h_j(\lambda,A_0) \in L_1\left(\mathbb{R}^d\right)\cap L_2\left(\mathbb{R}^d\right),$ when
$A_0<+\infty\,;$
\item $\tilde{h}_j(\lambda,A_0)\in L_2\left(\mathbb{R}^d\right),$ when $A_0=+\infty.$
\end{itemize}

Therefore,  substituting (\ref{K}) and  (\ref{h12}) into (\ref{g01}) and (\ref{g02}) and legitimately changing the order of integration, for $A_0<+\infty$  we  get 
\[I(A_0):=\int_{\mathbb{R}^{2d}}\frac{K_{\Delta }(\lambda _1+\lambda _2)\,h_1(\lambda_1,A_0)h_1(\lambda_2,A_0)}{\left\| \lambda
_1\right\| ^{(d-\alpha )/2} \left\| \lambda _2\right\| ^{(d-\alpha )/2}}\,\mathrm{d}\lambda _1\mathrm{d}\lambda _2\]
\begin{equation}\label{I}
=\int_{\Delta } \int_{\mathbb{R}^{2d}} \frac{e^{i(\lambda _1+\lambda _2,u)}h_1(\lambda_1,A_0)h_1(\lambda_2,A_0)}{\left\| \lambda
_1\right\| ^{(d-\alpha )/2} \left\| \lambda _2\right\| ^{(d-\alpha )/2}}\,\mathrm{d}\lambda _1\mathrm{d}\lambda _2\,\mathrm{d}u=\int_{\Delta}  \left(\hat{\tilde{h}}_1(u,A_0)\right)^2\,\mathrm{d}u,
\end{equation}
and similarly
\[
I_0(A_0):=\int_{\mathbb{R}^{4d}}\triangledown_K(\lambda _1,\lambda _2,\lambda _3,\lambda _4)\,h_1(\lambda_1,A_0)h_1(\lambda_2,A_0)h_2(\lambda_3,A_0)h_2(\lambda_4,A_0)\,\prod_{j=1}^4\frac{\mathrm{d}\lambda _j}{\left\| \lambda_j\right\| ^{(d-\alpha )/2}}
\]
\begin{equation}\label{I0}
=\int_{\Delta}  \left(\hat{\tilde{h}}_1(u,A_0)\right)^2\,\mathrm{d}u\int_{\Delta}  \left(\hat{\tilde{h}}_2(u,A_0)\right)^2\,\mathrm{d}u -\left(\int_{\Delta}  \hat{\tilde{h}}_1(u,A_0)\hat{\tilde{h}}_2(u,A_0)\,\mathrm{d}u\right)^2\,,
\end{equation}
where $\hat{\tilde{h}}_j(\cdot,A_0)$ are the Fourier transforms of $\tilde{h}_j(\cdot,A_0),$ $j=1,2.$

Notice, that for $A_0=+\infty$ we get, see \cite{ole12},
\begin{equation}\label{chi}\chi_{\raisebox{-3pt}{\scriptsize $v\left(\ell_j\{\Delta\}\right)$}}(u)=\left(\frac{\ell_j\{\Delta\}}{\sqrt{2\pi}}\right)^d\int_{\mathbb
R^d}e^{i(u,\lambda)}\frac{J_{{d}/{2}} (\left\| \lambda\right\|\ell_j\{\Delta\})}{
(\left\| \lambda\right\|\ell_j\{\Delta\})^{d/2}}\,\mathrm{d}\lambda=\hat{\tilde{h}}_j(u,+\infty),\quad j=1,2.\end{equation} 

By the definition (\ref{h12}), 
\[\tilde{h}_j(\lambda,A_0)\to \tilde{h}_j(\lambda,+\infty)\ \text{in} \ L_2(\mathbb R^d),\quad \text{when} \ A_0\to +\infty.\]
Hence, (\ref{I}), (\ref{I0}), and (\ref{chi})  yield
\[
I(A_0)\to \int_{\Delta}  \left(\hat{\tilde{h}}_1(u,+\infty)\right)^2\,\mathrm{d}u=\int_{\Delta}  \left(\chi_{\raisebox{-3pt}{\scriptsize $v\left(\ell_1\{\Delta\}\right)$}}(u)\right)^2\,\mathrm{d}u=|v(\ell_1\{\Delta\})\cap\Delta|\]
and 
\[
I_0(A_0)\to |v(\ell_1\{\Delta\})\cap\Delta|\cdot |v\left(\ell_2\{\Delta\}\right)\cap\Delta| -|v\left(\min\left(\ell_1\{\Delta\},\ell_2\{\Delta\}\right)\right)\cap\Delta|^2\,\] 
when $A_0\to +\infty.$

If $\ell_1\{\Delta\}:=diam\{\Delta\}$ and $\ell_2\{\Delta\}$ is a such radius that $|v(\ell_2\{\Delta\})\cap\Delta|=|\Delta|/2,$ then 
\[I(A_0)\to |\Delta|>0 \quad \text{and}\quad
I_0(A_0)\to |\Delta|^2/4>0.\] 
Hence, there exists  $A_0<+\infty$  such that the  conditions (\ref{g01}) and (\ref{g02}) are satisfied for $h_j(\lambda)=h_j(\lambda,A_0),$ $j=1,2.$ Finally, we complete the proof noting that $h_j(\lambda,A_0),$ $j=1,2,$ are linearly independent functions.
\end{proof}
\begin{definition} Let $Y_1$ and $Y_2$ be arbitrary random variables. The uniform (Kolmogorov) metric  for the distributions of $Y_1$ and $Y_2$ is defined by the formula
\begin{equation*}
{\rho}\left( Y_1,Y_2\right) =\underset{z\in \mathbb R}{\sup }\left| P\left( Y_1\leq
z\right) -P\left( Y_2\leq z\right) \right| .
\end{equation*}
\end{definition}
The following result follows from Lemma~1.8~\cite{pet}.

\begin{lemma}\label{lem1}
If $X,Y,$  and $Z$ are arbitrary random variables, then for any $\varepsilon >0:$
\begin{equation*}
\rho\left( X+Y,Z\right) \leq {\rho}( X,Z)+\rho\left( Z+\varepsilon,Z\right)+P\left( \left|
Y\right| \ge \varepsilon \right).
\end{equation*}
\end{lemma}

\section{Main result}\label{sec3}

In this paper we consider the case of Rosenblatt-type limit distributions, i.e. $\kappa=2$ and $\alpha \in (0,d/2)$  in Theorem~\ref{th2}. The main result describes the rate of convergence of $K_r$ to $X_2(\Delta),$ when $r\rightarrow \infty.$ To prove it we use some techniques and facts from \cite{bra,anh,mink}.

\begin{theorem}\label{rate}
If Assumptions~{\rm \ref{ass1}} and {\rm \ref{ass2}} hold,   $q< {d}/{2}-\alpha,$ and $H {\rm rank}\,G=2,$ then for any $\varkappa<\frac{1}{3}\min\left(\frac{\alpha(d-2\alpha)}{d-\alpha},\varkappa_1\right),$
\begin{equation*}
\rho\left(\frac{2\,K_{r}}{C_2\,r^{d-\alpha}L(r)},X_{2}(\Delta) \right)=o (r^{-\varkappa}),\quad r\rightarrow \infty ,
\end{equation*}
where  $C_2$ is defined by {\rm(\ref{herm})}  and
\[\varkappa_1:=2\min\left(q,\left(\frac{2}{d-2\alpha}+\frac{2}{d+1-2\alpha}+\frac{1}{\upsilon}\right)^{-1}\right).\]
\end{theorem}
\begin{rem} The order of convergence $\varkappa$ depends on the three parameters $\alpha,$ $\upsilon,$ and $q.$ Recall the meaning of these parameters: $\alpha$ is a long-range dependence parameter,  $q$ gives the order for the upper bound of the slowly varying (with remainder) function $L(\cdot),$ and $\upsilon$ describes the  magnitude of deviations of the spectral density from $c_2(d,\alpha )\left\| \lambda \right\|^{\alpha
-d}L\left( 1/\left\| \lambda \right\|\right)$ at the origin. 
\end{rem}
\begin{proof} Since $H {\rm rank}\,G=2,$ it follows that $K_r$ can be represented in the space of squared-integrable random variables $L_2(\Omega)$ as
\[K_r = K_{r,2}+S_r :=\frac{C_2 }{2}
\int_{\Delta(r)}H_2 (\eta (x))\,\mathrm{d}x + \sum_{j\geq 3}\frac{C_j}{j!}
\int_{\Delta(r)}H_j (\eta (x))\,\mathrm{d}x,
\]
where $C_j $ are coefficients of the Hermite series (\ref{herm}) of the function $G(\cdot).$

Notice that $\mathbf{E}K_{r,2}=\mathbf{E}S_r=\mathbf{E} X_{2}(\Delta)=0,$ and
\[X_{r,2}=\frac{2\,K_{r,2}}{C_2\,r^{d-\alpha}L(r)}.\]

It follows from Assumption~\ref{ass1} that $|L(u)/u^\alpha|=|B(u)|\le B(0)=1.$ Thus, by the proof of Theorem~4~\cite{mink}, 
\begin{eqnarray}
 \mathbf{Var}\, S_{r} &\leq& |\Delta|^2r^{2d-3\alpha}\sum_{j\geq 3}\frac{C_j ^2}{j !}
   \int_0^{diam\left\{ \Delta
   \right\}} z^{-3\alpha} L^{3}\left(rz\right)\psi _{\Delta}(z)dz\nonumber\\
 &=&  |\Delta|^2r^{2(d-\alpha)}L^{2}(r)\sum_{j\geq 3}\frac{C_j ^2}{j !}
    \int_0^{diam\left\{ \Delta
    \right\}} z^{-2\alpha} \frac{L^{2}\left(rz\right)}{L^{2}(r)}  \frac{L\left(rz\right)}{(rz)^{\alpha}}\psi _{\Delta}(z)\,\mathrm{d}z.\label{varq}
\end{eqnarray}

We represent the integral in (\ref{varq}) as the sum of two integrals $I_1$ and $I_2$ with the ranges of integration $[0,r^{-\beta_1}]$ and $(r^{-\beta_1},diam\left\{ \Delta
    \right\}]$ respectively, where $\beta_1\in(0,1).$

It follows from Assumption~\ref{ass1} that  $|L(u)/u^\alpha|=|B(u)|\le B(0)=1$ and we can estimate the first integral as 
\[I_1\le\int_0^{r^{-\beta_1}} z^{-2\alpha} \frac{L^{2}\left(rz\right)}{L^{2}(r)}  \psi _{\Delta}(z)\,\mathrm{d}z
\le\left(\frac{\sup_{0\le s\le r}s^{\delta/2}L\left(s\right)}{r^{\delta/2}L(r)}\right)^{2} \int_0^{r^{-\beta_1}} z^{-\delta}z^{-2\alpha}  \psi _{\Delta}(z)\,\mathrm{d}z,
\]
where $\delta$ is an arbitrary number in $(0,\min(\alpha,d-2\alpha)).$

By Assumption~\ref{ass1} the function $L\left(\cdot\right)$ is locally bounded. By Theorem~1.5.3 \cite{bin},  there exist $r_0>0$ and $C>0$ such that for all $r\ge r_0$
\[\frac{\sup_{0\le s\le r}s^{\delta/2}L\left(s\right)}{r^{\delta/2}L(r)}\le C.\]

Using (\ref{dint}) we obtain
\[\int_0^{r^{-\beta_1}} z^{-\delta}z^{-2\alpha}  \psi _{\Delta}(z)\,\mathrm{d}z\le\frac{C}{\left| \Delta \right| }\int_{0}^{r^{-\beta_1}}\rho^{d-2\alpha-1-\delta}\,\mathrm{d}\rho=\frac{C\,r^{-\beta_1(d-2\alpha -\delta)}}{(d-2\alpha-\delta)\left|\Delta \right|}.
\]

Applying Theorem~1.5.3 \cite{bin}  we get
\[I_2\le
\frac{\sup_{r^{1-\beta_1}\le s\le r\cdot diam\left\{ \Delta
    \right\}}s^{\delta}L^{2}\left(s\right)}{r^{\delta}L^{2}(r)}\cdot \sup_{r^{1-\beta_1}\le s\le r\cdot diam\left\{ \Delta
\right\}}\frac{L\left(s\right)}{s^\alpha}\cdot \int_0^{diam\left\{ \Delta \right\}} z^{-(\delta+2\alpha)}  \psi _{\Delta}(z)\,\mathrm{d}z\]
\[\le C\cdot o(r^{-(\alpha-\delta)(1-\beta_1)}),
 \]
when $r$ is sufficiently large.

Notice that by (\ref{par})
\[\sum_{j\geq 3}\frac{C_j ^2}{j !}\le \int_{\mathbb R}G^2(w)\ \phi ( w )\,\mathrm{d}w< +\infty.\]

Hence, for sufficiently large $r$
\[
 \mathbf{Var}\, S_{r} \leq C\,r^{2(d-\alpha)}L^{2}(r)\left( r^{-\beta_1(d-2\alpha-\delta)}
+
o\left(r^{-(\alpha-\delta)(1-\beta_1)}\right)\right).
\]
Choosing $\beta_1=\frac{\alpha}{d-\alpha}$ to minimize the upper bound we get
\[
 \mathbf{Var}\, S_{r} \leq C r^{2(d-\alpha)}L^{2}(r)r^{-\frac{\alpha(d-2\alpha)}{d-\alpha}+\delta}.
\]

It follows from Lemma~\ref{lem3} that 

\[{\rho}\left(X_2(\Delta)+\varepsilon,X_2(\Delta)\right) \le \varepsilon\sup_{z\in \mathbb R}\, p_{X_2(\Delta)}\left( z\right)\le \varepsilon\, C.\]

Applying Chebyshev's inequality and Lemma~\ref{lem1} to $X=X_{r,2},$ $Y=\frac{2\,S_r}{C_2\,r^{d-\alpha}L(r)},$ and $Z=X_{2}(\Delta),$ we get

$${\rho}\left( \frac{2\,K_{r}}{C_2\,r^{d-\alpha}L(r)},X_2(\Delta)\right)={\rho}\left( X_{r,2}+\frac{2\,S_r}{C_2\,r^{d-\alpha}L(r)},X_2(\Delta)\right)$$
\[
\le {\rho}\left(X_{r,2},X_2(\Delta)\right)+C\left(\varepsilon+ \varepsilon^{-2}\,r^{-\frac{\alpha(d-2\alpha)}{d-\alpha}+\delta}\right),
\]
for sufficiently large $r.$

Choosing $\varepsilon:=r^{-\frac{\alpha(d-2\alpha)}{3(d-\alpha)}}$ to minimize the second term we obtain
\begin{equation}\label{up11}
{\rho}\left( \frac{2\,K_{r}}{C_2\,r^{d-\alpha}L(r)},X_2(\Delta)\right)\le {\rho}\left(X_{r,2},X_2(\Delta)\right)+C\,r^{-\frac{\alpha(d-2\alpha)}{3(d-\alpha)}+\delta}.
\end{equation}

Applying Lemma~\ref{lem1} once again to $X=X_{2}(\Delta),$ $Y=X_{r,2}-X_{2}(\Delta),$ and $Z=X_{2}(\Delta)$ we obtain
\begin{eqnarray}
\rho\left(X_{r,2},X_{2}(\Delta) \right) &\leq &\varepsilon_1\,C
+P\left\{ \left|X_{r,2}-X_2(\Delta)\right| \geq \varepsilon_1 \right\}  \notag \\
&\leq &\varepsilon_1\,C +\varepsilon_1^{-2} \mathbf{Var}\left(X_{r,2}-X_2(\Delta)\right).
\label{38}
\end{eqnarray}

Below we show how to estimate $\mathbf{Var}\left(X_{r,2}-X_2(\Delta)\right).$

By the self-similarity of Gaussian white noise and  formula (2.1) \cite{dob}
\begin{eqnarray*}
&&X_{r,2}\stackrel{\mathcal{D}}{=}
c_2(d,\alpha)\int_{\mathbb{R}^{2d}}^{{\prime }}K_{\Delta}(\lambda _1+\lambda _2)Q_r(\lambda _1,\lambda _2)\frac{W(\mathrm{d}\lambda
_1) W(\mathrm{d}\lambda _2)}{\left\| \lambda _1\right\| ^{(d-\alpha )/2}
\left\| \lambda _2\right\| ^{(d-\alpha )/2}},
\end{eqnarray*}
where
\[
Q_r(\lambda _1,\lambda _2): =r^{\alpha-d}L^{-1}(r)\
c_2^{-1}(d,\alpha)  \left[ \left\| \lambda _1\right\|^{d-\alpha} \left\| \lambda _2\right\|^{d-\alpha} f\left( \frac{\left\| \lambda _1\right\| }r\right) f\left( \frac{\left\| \lambda _2\right\| }r\right)\right] ^{1/2}.
\]

Notice that
\[
X_2(\Delta) =c_2(d,\alpha )  \int_{\mathbb{R}^{2d}}^{{\prime }}K_{\Delta}(\lambda _1+\lambda _2) \frac{W(d\lambda _1) W(d\lambda _2)}{\left\| \lambda
_1\right\| ^{(d-\alpha )/2} \left\| \lambda _2\right\| ^{(d-\alpha )/2}}.\]

By the isometry property of multiple stochastic integrals

\[R_r:=\frac{\mathbb{E}\left| X_{r,2}-X_2(\Delta)\right|^2}{c_2^{2}(d,\alpha)}=\int_{\mathbb{R}^{2d}}\frac{|K_{\Delta}(\lambda _1+\lambda _2)|^2\left(Q_r(\lambda_1,\lambda_2)-1\right)^2}{\left\| \lambda _1\right\| ^{d-\alpha}\left\| \lambda _2\right\| ^{d-\alpha}}\,\mathrm{d}\lambda
_1 \,\mathrm{d}\lambda _2. \]

Let us rewrite the integral $R_r$ as the sum of two integrals $I_3$ and $I_4$ with the regions $A(r):=\{(\lambda _1,\lambda _2)\in\mathbb{R}^{2d}:\ \max(||\lambda _1||, ||\lambda _2||)\le r^\gamma\}$ and $\mathbb{R}^{2d}\setminus A(r)$ respectively, where $\gamma\in(0,1).$ Our intention is to use the monotone equivalence property of regularly varying functions in the regions $A(r).$

First we consider the case of $(\lambda _1,\lambda _2)\in A(r).$  By Assumption~\ref{ass2} and the inequality $|\sqrt{ab}-1|\le |1-a|+|1-b|,$ for sufficiently large $r,$ we obtain

\[
|Q_r(\lambda _1,\lambda _2)-1| \le \left|1-\frac{L\left( \frac{r}{\left\| \lambda _1\right\| }\right)}{L(r)} \right| + \left|1-\frac{L\left( \frac{r}{\left\| \lambda _2\right\| }\right)}{L(r)} \right| + C\frac{L\left( \frac{r}{\left\| \lambda _1\right\| }\right)}{L(r)} \left( \frac{\left\| \lambda _1\right\| }r\right)^{\upsilon}\]
\[ +\, C\,\frac{L\left( \frac{r}{\left\| \lambda _2\right\| }\right)}{L(r)} \left( \frac{\left\| \lambda _2\right\| }r\right)^{\upsilon}.
\]

By Lemma~\ref{lem0}, if $||\lambda _j||\in (1, r^\gamma),$ $j=1,2,$ then for arbitrary $\beta_2>0$ and sufficiently large $r$ we get
\[\left|1-\frac{L\left( \frac{r}{\left\| \lambda _j\right\| }\right)}{L(r)} \right|=\frac{L\left( \frac{r}{\left\| \lambda _j\right\| }\right)}{L(r)} \cdot\left|1-\frac{L(r)}{L\left( \frac{r}{\left\| \lambda _j\right\| }\right)} \right|\le C\,\frac{L\left( \frac{r}{\left\| \lambda _j\right\| }\right)}{L(r)} \cdot \frac{\left\| \lambda _j\right\|^{q}}{r^q} \]
\[\le C\,  \frac{\left\| \lambda _j\right\|^{q+\beta_2}}{r^q}\cdot\frac{\sup_{||\lambda _j||\in (1, r^\gamma)}\left( \frac{r}{\left\| \lambda _j\right\| }\right)^{\beta_2} L\left( \frac{r}{\left\| \lambda _j\right\| }\right)}{r^{\beta_2}\,L(r)}
\le C\,  \frac{\left\| \lambda _j\right\|^{q+{\beta_2}}}{r^q}\]
\begin{equation}\label{gr1}\times\frac{\sup_{z\in (0,r)} z^{\beta_2} L\left( z\right)}{r^{\beta_2}\,L(r)}
 \le C\,  \frac{\left\| \lambda _j\right\|^{q+{\beta_2}}}{r^q}. \end{equation}

By  Lemma~\ref{lem0} for $||\lambda _j||\le 1,$ $j=1,2,$ we obtain
\begin{equation}\label{gr2}
\left|1-\frac{L\left( \frac{r}{\left\| \lambda _j\right\| }\right)}{L(r)} \right|\le \frac{C}{r^q}.
\end{equation}

Hence, by  (\ref{gr1}) and (\ref{gr2})
\[
|Q_r(\lambda _1,\lambda _2)-1|^2 \le C \left( \frac{\left\| \lambda _1\right\| }r\right)^{2\upsilon}\cdot\frac{L^2\left( \frac{r}{\left\| \lambda _1\right\| }\right)}{L^2(r)} + C\left( \frac{\left\| \lambda _2\right\| }r\right)^{2\upsilon}\cdot\frac{L^2\left( \frac{r}{\left\| \lambda _2\right\| }\right)}{L^2(r)} 
\]
\[+\,C\,r^{-2q}  \left(\left\| \lambda _1\right\|^{(\mu_1+1)(q+{\beta_2})}+ \left\| \lambda _2\right\|^{(\mu_2+1)(q+{\beta_2})}\right)
\]
for $(\lambda _1,\lambda _2)\in A(r)\cap B_\mu,$ where
\[B_\mu :=\{(\lambda _1,\lambda _2)\in\mathbb{R}^{2d}: ||\lambda_j||\le 1,\ \mbox{if}\ \mu_j=-1,\ \mbox{and}\ ||\lambda_j||> 1, \ \mbox{if}\ \mu_j=1, j=1,2\}, \]
$\mu=(\mu_1,\mu_2)\in \{-1,1\}^2$ is a binary vector of length $2.$

By Lemma~\ref{finint}, for  $r>1$
\[r^{-2q} \int_{A(r)\cap B_{(\mu_1,-1)}}\frac{|K_{\Delta}(\lambda _1+\lambda _2)|^2\,\mathrm{d}\lambda
_1 \mathrm{d}\lambda _2 }{\left\| \lambda _1\right\| ^{d-\alpha}\left\| \lambda _2\right\| ^{d-\alpha}}\le \frac{C}{r^{2q}},\]
when $\mu_1\in \{-1,1\}.$

Similarly, by Lemma~\ref{finint}  we obtain 
\[r^{-2q} \int_{A(r)\cap B_{(\mu_1,1)}}\frac{|K_{\Delta}(\lambda _1+\lambda _2)|^2\,\mathrm{d}\lambda
_1 \mathrm{d}\lambda _2 }{\left\| \lambda _1\right\| ^{d-\alpha}\left\| \lambda _2\right\| ^{d-\alpha-2q-2\beta_2}}\le \frac{C}{r^{2q}},\]
when  $\mu_1\in \{-1,1\},$ $q \in (0,d/2-\alpha),$ $\alpha>0,$ and $\beta_2$ is sufficiently small.

By  properties of slowly varying functions   \cite[Theorem~1.5.3]{bin}
\[\lim_{r\to \infty}\frac{\sup_{||\lambda _j||\le 1}\left(\frac{r}{||\lambda _j||}\right)^{-\upsilon}L\left(\frac{r}{||\lambda _j||}\right)}{r^{-\upsilon}L(r)}=\lim_{r\to \infty}\frac{\sup_{z\ge r}z^{-\upsilon}L\left(z\right)}{r^{-\upsilon}L(r)}=1.\]
Hence, it holds for sufficiently large $r$ that 
\[ \left( \frac{\left\| \lambda _j\right\| }r\right)^{2\upsilon}\cdot\frac{L^2\left( \frac{r}{\left\| \lambda _j\right\| }\right)}{L^2(r)}\le C\, r^{-2\upsilon}.\]

Therefore, by Lemma~\ref{finint}  we obtain for sufficiently large $r$
\[ I_3\le C\,r^{-2q}\sum_{\mu\in \{-1,1\}^2} \int_{A(r)\cap B_\mu}\frac{|K_{\Delta}(\lambda _1+\lambda _2)|^2\,\mathrm{d}\lambda
_1 \mathrm{d}\lambda _2 }{\left\| \lambda _1\right\| ^{d-\alpha}\left\| \lambda _2\right\| ^{d-\alpha-(\mu_2+1)(q+{\beta_2})}}\]
\begin{equation}\label{del}+\, C\sup_{\mu_2\in \{-1,1\}}
\sup_{\left\| \lambda _2\right\|\le r^\gamma}\frac{\left\| \lambda _2\right\|^{2\upsilon+\mu_2\delta}}{r^{2\upsilon}}
\le C\,r^{-2q}+ C\,r^{-2\upsilon(1-\gamma)+\delta}.
\end{equation}

It follows from Assumption~\ref{ass2} and the specification  of the estimate (23)  in the proof of Theorem~5 \cite{mink} for $k=2$ that for each positive $\delta$ there exists $r_0>0$ such that for all $r\ge r_0,$   $(\lambda _1,\lambda _2)\in B_{(1,\mu_2)},$ and $\mu_2\in \{-1,1\},$ it holds
\[
\frac{|K_{\Delta}(\lambda _1+\lambda _2)|^2\left(Q_r(\lambda _1,\lambda _2)-1\right)^2}{\left\| \lambda _1\right\| ^{d-\alpha}
\left\| \lambda _2\right\| ^{d-\alpha}}\le \frac{C\, |K_{\Delta}(\lambda _1 +\lambda _2)|^2}{\left\| \lambda _1\right\| ^{d-\alpha}
\left\| \lambda _2\right\| ^{d-\alpha}}
+C\,\frac{|K_{\Delta}(\lambda _1+\lambda _2)|^2}{\left\| \lambda _1\right\| ^{d-\alpha-\delta}
\left\| \lambda _2\right\| ^{d-\alpha-\mu_2\delta}}.
\]

Hence, we can estimate $I_4$ as shown below
\[ I_4\le 2\int_{\mathbb{R}^{d}}\int_{||\lambda _1||> r^\gamma}\frac{|K_{\Delta}(\lambda _1+\lambda _2)|^2\left(Q_r(\lambda_1,\lambda_2)-1\right)^2\mathrm{d}\lambda
_1 \mathrm{d}\lambda _2}{\left\| \lambda _1\right\| ^{d-\alpha}\left\| \lambda _2\right\| ^{d-\alpha}}\]
\[ \le C\int_{\mathbb{R}^{d}}\int_{||\lambda _1||> r^\gamma}\frac{|K_{\Delta}(\lambda _1+\lambda _2)|^2\,\mathrm{d}\lambda
_1 \mathrm{d}\lambda _2}{\left\| \lambda _1\right\| ^{d-\alpha}\left\| \lambda _2\right\| ^{d-\alpha}}\]
\[+\,  C\sum_{\mu_2\in\{1,-1\}}\int_{||\lambda_2||^{\mu_2}\ge 1}\int_{||\lambda _1||> r^\gamma}\frac{|K_{\Delta}(\lambda _1+\lambda _2)|^2}{\left\| \lambda _1\right\| ^{d-\alpha-\delta}\left\| \lambda _2\right\| ^{d-\alpha-\mu_2\delta}}\,\mathrm{d}\lambda
_1 \mathrm{d}\lambda_2\]
\[\le C\max_{\mu_2\in\{0,1,-1\}}\int_{\mathbb{R}^{d}}\int_{||\lambda _1||> r^\gamma}\frac{|K_{\Delta}(u)|^2}{\left\| \lambda _1\right\| ^{d-\alpha-\delta}\left\|u- \lambda _1\right\| ^{d-\alpha-\mu_2\delta}}\,\mathrm{d}\lambda
_1 \mathrm{d}u\]
\[= C\max_{\mu_2\in\{0,1,-1\}}\int_{\mathbb{R}^{d}}
\frac{|K_{\Delta}(u)|^2}{\left\| u\right\| ^{d-2\alpha-(\mu_2+1)\delta}}
\int_{\|\lambda _1\|> \frac{r^\gamma}{\|u\|}}\frac{\,\mathrm{d}\lambda
_1\mathrm{d}u}{\left\| \lambda _1\right\| ^{d-\alpha-\delta}\left\|\frac{u}{\|u\|}- \lambda _1\right\| ^{d-\alpha-\mu_2\delta}}.\]

Taking into account that for $\delta\in (0,\min(\alpha,{d}/{2}-\alpha))$
\[\sup_{u\in\mathbb{R}^{d}\setminus\{0\}}\int_{\mathbb{R}^{d}}\frac{\,\mathrm{d}\lambda
_1}{\left\| \lambda _1\right\| ^{d-\alpha-\delta}\left\|\frac{u}{\|u\|}- \lambda _1\right\| ^{d-\alpha-\mu_2\delta}}\le C,\]
we obtain
\[I_4\le C\max_{\mu_2\in\{0,1,-1\}}\int_{\|u\|\le r^{\gamma_0}}
\frac{|K_{\Delta}(u)|^2}{\left\| u\right\| ^{d-2\alpha-(\mu_2+1)\delta}}
\int_{\|\lambda _1\|> r^{\gamma-\gamma_0}}\frac{\,\mathrm{d}\lambda
_1\mathrm{d}u}{\left\| \lambda _1\right\| ^{d-\alpha-\delta}\left\|\frac{u}{\|u\|}- \lambda _1\right\| ^{d-\alpha-\mu_2\delta}}\]
\[+\, C\max_{\mu_2\in\{0,1,-1\}}\int_{\|u\|> r^{\gamma_0}}
\frac{|K_{\Delta}(u)|^2\,\mathrm{d}u}{\left\| u\right\| ^{d-2\alpha-(\mu_2+1)\delta}},
\]
where $\gamma_0\in (0,\gamma).$

By Lemma~\ref{finint}, there exists $r_0>0$ such that for all $r\ge r_0$ the first summand is bounded by
\[C\max_{\mu_2\in\{0,1,-1\}}\int_{\mathbb{R}^{d}}
\frac{|K_{\Delta}(u)|^2\,\mathrm{d}u}{\left\| u\right\| ^{d-2\alpha-(\mu_2+1)\delta}}
\int_{\|\lambda _1\|> r^{\gamma-\gamma_0}}\frac{\,\mathrm{d}\lambda
_1}{\left\| \lambda _1\right\|^{2d-2\alpha-\delta-\mu_2\delta}}\le Cr^{-(\gamma-\gamma_0)(d-2\alpha-2\delta)}.\]

Therefore, for sufficiently large $r,$
\begin{equation}\label{+} I_4\le Cr^{-(\gamma-\gamma_0)(d-2\alpha-2\delta)}+C\int_{\|u\|> r^{\gamma_0}}
\frac{|K_{\Delta}(u)|^2\,\mathrm{d}u}{\left\| u\right\| ^{d-2\alpha-2\delta}}.
\end{equation}

By the spherical $L_2$-average decay rate of the Fourier transform \cite{bra} for $\delta<d+1-2\alpha$ and sufficiently large $r$ we get the following estimate of the integral in (\ref{+})
\[\int_{\|u\|> r^{\gamma_0}}
\frac{|K_{\Delta}(u)|^2\,\mathrm{d}u}{\left\| u\right\| ^{d-2\alpha-2\delta}}\le 
C\int_{z> r^{\gamma_0}}\int_{S^{d-1}}
\frac{|K_{\Delta}(z\omega)|^2}{z^{1-2\alpha-2\delta}}\,\mathrm{d}\omega \mathrm{d}z\]
\begin{equation}\label{uppp}
\le C\int_{z> r^{\gamma_0}} \frac{\,\mathrm{d}z}{z^{d+2-2\alpha-2\delta}}=
C\,r^{-\gamma_0(d+1-2\alpha-2\delta)},
\end{equation}

where $S^{d-1}:=\{x\in \mathbb{R}^{d}:\left\Vert x\right\Vert =1\}$  is a sphere
of radius 1 in $\mathbb{R}^d.$

Combining estimates (\ref{up11}), (\ref{38}), (\ref{del}), (\ref{+}), (\ref{uppp}), and choosing $\varepsilon_1:=r^{-\beta},$  we obtain
\[{\rho}\left( \frac{2\,K_{r}}{C_2\,r^{d-\alpha}L(r)},X_2(\Delta)\right)\le C\left(r^{-\frac{\alpha(d-2\alpha)}{3(d-\alpha)}+\delta}+ r^{-\beta} +
r^{-2\upsilon(1-\gamma)+2\beta+\delta}+\,r^{-2q+2\beta}\right.\]
\[ \left.+r^{-(\gamma-\gamma_0)(d-2\alpha-2\delta)+2\beta}+r^{-\gamma_0(d+1-2\alpha-2\delta)+2\beta}\right).
\]
Therefore, for any  $\tilde\varkappa_1\in (0,3\varkappa_0)$ one can choose a sufficiently small $\delta>0$ such that 
\begin{equation}\label{bou}
{\rho}\left( \frac{2\,K_{r}}{C_2\,r^{d-\alpha}L(r)},X_2(\Delta)\right)\le Cr^\delta\left(r^{-\frac{\alpha(d-2\alpha)}{3(d-\alpha)}}+ r^{-\frac{\tilde\varkappa_1}{3}}\right),
\end{equation}
where  \[\varkappa_0:=\sup_{\stackrel{\beta>0}{\stackrel{\gamma\in(0,1)}{\gamma_0\in (0,\gamma)}}}\min\left(\beta,
2\upsilon(1-\gamma)-2\beta,2q-2\beta,(\gamma-\gamma_0)(d-2\alpha)-2\beta,
\gamma_0(d+1-2\alpha)-2\beta\right).
\]
Note, that for fixed $\gamma\in(0,1)$
\[\sup_{\gamma_0\in (0,\gamma)}\min\left((\gamma-\gamma_0)(d-2\alpha),
\gamma_0(d+1-2\alpha)\right)=\frac{\gamma(d-2\alpha)(d+1-2\alpha)}{2d+1-4\alpha},\]
and 
\[\sup_{\gamma\in(0,1)}\min\left(2\upsilon(1-\gamma),\frac{\gamma(d-2\alpha)(d+1-2\alpha)}{2d+1-4\alpha}\right)=2\left(\frac{2}{d-2\alpha}+\frac{2}{d+1-2\alpha}+\frac{1}{\upsilon}\right)^{-1}.\]

Thus,  $\varkappa_0=\sup_{\beta>0}\min\left(\beta,
\varkappa_1-2\beta\right)=\varkappa_1/3.$

Finally, by (\ref{bou}) and $\tilde\varkappa_1<\varkappa_1$ we obtain the statement of the theorem. 
\end{proof}

\begin{rem} The obtained results and the methods of \cite[\S 5.5]{pet} provide a theoretical framework for generalizations  to non-uniform estimates of the remainder in the non-central limit theorem.\end{rem}

\section{Examples}\label{sec4}
Theorem~\ref{rate} was proven under rather general assumptions. In  this section we present some examples and specifications of the  results of Sections~\ref{sec2} and \ref{sec3}.

\begin{example}
 If $\Delta $ is the ball $v(1)$ in $\mathbb{R}^{d},$ then
\[
K_{v(1)}(x)=\int_{v(1)}e^{i(x,u)} du=(2\pi )^{d/2}
\frac{J_{d/2}(\|x\|)}{\left\|x\right\|^{d/2}},\quad
x\in\mathbb{R}^d,
\]
and  we obtain
\[
X_2(v(1))=(2\pi )^{d/2}c_2(d,\alpha )\int_{\mathbb{R}^{2d}}^{{\prime }}\frac{J_{d/2}(\left\| \lambda _1
+\lambda _2\right\| )}{\left\| \lambda _1 +\lambda _2\right\|
^{d/2}} \frac{W(\mathrm{d}\lambda _1) W(\mathrm{d}\lambda _2)}{\left\| \lambda
_1\right\| ^{(d-\alpha )/2} \left\| \lambda _2\right\| ^{(d-\alpha )/2}%
}\,.\]
\end{example}
\begin{example} If $\Delta $ is the multidimensional
rectangle   $[\mathbf{a},\mathbf{b}]:=\{x\in \mathbb{R
}^{d}: x_j\in [a_j,b_j],\ j=1,...,d\}$ and $a_j<0<b_j,$ $j=1,...,d,$ then
\[
K_{[\mathbf{a},\mathbf{b}]}\left( x \right)=\int_{[\mathbf{a},\mathbf{b}]}e^{i(x,u)} du
=\prod_{j=1}^{d}\frac{e^{ib_{j}x_{j}}-e^{ia_{j}x_{j}}}{%
ix_{j}}
\]
and
\[
X_2([\mathbf{a},\mathbf{b}])=c_2(d,\alpha )\int_{\mathbb{R}^{2d}}^{{\prime }}
\prod_{j=1}^{d}\frac{\left(e^{ib_{j}(\lambda_{1j}
+\lambda_{2j})}-e^{ia_{j}(\lambda_{1j}
+\lambda_{2j})}\right)}{%
i(\lambda_{1j}
+\lambda_{2j})} \frac{W(\mathrm{d}\lambda_{1})\, W(\mathrm{d}\lambda_{2})}{\left\| \lambda
_1\right\| ^{(d-\alpha )/2} \left\| \lambda _2\right\| ^{(d-\alpha )/2}}
,\]
where $\lambda_{m}=\left(\lambda_{m1},...,\lambda_{md}\right),$ $m=1,2.$
\end{example}

\begin{example}\label{ex5} Let us consider $\eta (x),$ $x\in\mathbb{R}^d,$ with the covariance function of the form
\begin{equation}\label{lgm}
B\left( \left\Vert x\right\Vert\right) =
\begin{cases} 1-\frac{\alpha}{\theta+\alpha}\left\Vert x\right\Vert^\theta,& \left\Vert x\right\Vert\le 1;\\
\frac{\theta}{\theta+\alpha}\left\Vert x\right\Vert^{-\alpha},& \left\Vert x\right\Vert > 1,
\end{cases}
\end{equation}
which was proposed as a local-global distinguisher model in \cite{gne}. It was shown in \cite{gne} that (\ref{lgm}) is a valid correlation function when $\alpha>0,$ $\theta\in(0,(3-d)/2],$ $d=1,2.$

The local-global distinguisher model obviously satisfies Assumption~\ref{ass1} and the condition~(\ref{gr}) with $L(t)=\frac{\theta}{\theta+\alpha},$ $t>1.$ 

In the case of long-memory stochastic processes, i.e. $d=1$ and $\alpha\in (0,1/2),$ taking the inverse Fourier transform of $B(\cdot)$ we obtain 

\[f(\left\| \lambda \right\|)=\frac{1}{{\pi }} \left(\frac{\sin \left( \left\| \lambda \right\| \right) }{\left\| \lambda \right\|}+ \frac{\theta}{\theta+\alpha}\left((\alpha-1)^{-1} \,{}_1F_2\left(\frac{1}{2}-\frac{\alpha}{2};\frac{1}{2},\frac{3}{2}-\frac{\alpha}{2};-\frac{\left\| \lambda \right\|^2}{4}\right)\right.\right.\]
\[\left.\left.+\left\| \lambda \right\|^{\alpha-1} \sin \left({\pi  \alpha}/{2}\right) \Gamma (1-\alpha)\right)-\frac{\alpha}{(\theta+1) (\theta+\alpha)}  \,{}_1F_2\left(\frac{\theta}{2}+\frac{1}{2};\frac{1}{2},\frac{\theta}{2}+\frac{3}{2};-\frac{\left\| \lambda \right\|^2}{4}\right)\right),
\]
where ${}_1F_2(a;b_1,b_2;z)$ is the generalized hypergeometric function \cite[\S 9.14]{gra} defined by 
\[{}_1F_2(a;b_1,b_2;z)=\sum_{j=0}^\infty\frac{(a)_j}{j!\,(b_1)_j(b_2)_j}\,z^j\]
$(a)_j=\Gamma(a+j)/\Gamma(a),$ $j\in\mathbb{N}_0:=\mathbb{N}\cup \{0\},$ $-a\not\in\mathbb{N}_0.$


The power series expansion of $f(\left\| \lambda \right\|)$ gives
\[f(\left\| \lambda \right\|)= \frac { \theta\,\Gamma  \left(1- \alpha \right)\sin \left( \pi \alpha/2 \right) }{\pi(\theta+\alpha)}\left\| \lambda \right\|^{\alpha-1}+\frac{\alpha\,\theta}{\pi(1+\theta)(\alpha-1)}+{\cal O}(1)\]
\[= c_2(1,\alpha )\cdot\frac{\theta}{\theta+\alpha}\cdot\left\| \lambda \right\|^{\alpha-1}+\left\| \lambda \right\|^{\alpha-1}\cdot{\cal O}(\left\| \lambda \right\|^{1-\alpha}),\]
when $\left\| \lambda \right\|\to 0.$
Hence,   $\upsilon=1-\alpha.$

Now we consider the case $\left\| \lambda \right\| \to +\infty.$ By the asymptotic expansion of ${}_1F_2\left(a;b_1,b_2;z\right)$ in trigonometric form \cite{wol} we obtain
\[\,{}_1F_2\left(\frac{1}{2}-\frac{\alpha}{2};\frac{1}{2},\frac{3}{2}-\frac{\alpha}{2};-\frac{\left\| \lambda \right\|^2}{4}\right)=\frac{\sqrt{\pi}\,\Gamma\left(\frac{3}{2}-\frac{\alpha}{2}\right)}{2^{\alpha-1}\Gamma\left(\alpha/2\right)}\,\left\| \lambda \right\|^{\alpha-1}\left(1+{\cal O}\left(\left\| \lambda \right\|^{-2}\right)\right)+{\cal O}\left(\left\| \lambda \right\|^{-1}\right),
\]

\[\,{}_1F_2\left(\frac{\theta}{2}+\frac{1}{2};\frac{1}{2},\frac{\theta}{2}+\frac{3}{2};-\frac{\left\| \lambda \right\|^2}{4}\right)={\cal O} \left(\left\| \lambda \right\|^{-1}\right).\]

Therefore, $\frac{\sqrt{\pi}\,\Gamma\left(\frac{3}{2}-\frac{\alpha}{2}\right)}{2^{\alpha-1}(1-\alpha)\Gamma\left(\alpha/2\right)}=\sin \left({\pi  \alpha}/{2}\right) \Gamma (1-\alpha)$ implies $f(\left\| \lambda \right\|)={\cal O}\left(1/\left\| \lambda \right\|\right),$ when $\left\| \lambda \right\| \to +\infty.$
By Remark~\ref{rem6}  Assumption~{\rm\ref{ass2}} holds true for the local-global distinguisher processes. 
%
%
\end{example}
\begin{example}\label{ex6} Assume there exists  $t_0>0$ such that $L(t)=a_0$ for all $t\ge t_0$ in Assumption~\ref{ass2}. 

Since the parameter $q $ is arbitrary in (\ref{gr}), the only condition on $q$ in Theorem~\ref{rate} is $q \in (0,{d}/{2}-\alpha).$
As a consequence, taking $q$ arbitrarily close to ${d}/{2}-\alpha$  makes
$q>\left(2{(d-2\alpha)^{-1}}+2(d+1-2\alpha)^{-1}+\upsilon^{-1}\right)^{-1},$ because $\left(2{(d-2\alpha)^{-1}}+2(d+1-2\alpha)^{-1}+\upsilon^{-1}\right)^{-1}$ $<{d}/{2}-\alpha.$ Thus, we get in the
definition of $\varkappa_1$ in Theorem~\ref{rate}
\begin{equation}\label{kappa_1}\varkappa_1=2\left(2{(d-2\alpha)^{-1}}+2(d+1-2\alpha)^{-1}+\upsilon^{-1}\right)^{-1}.
\end{equation}

For  instance, let us consider Example~\ref{ex5} with $d=1$ and $v=1-\alpha.$ Figure~\ref{fig1} displays the graphs of $\frac{\alpha(d-2\alpha)}{3(d-\alpha)}$ and $\varkappa_1/3,$ plotted as  functions of the variable $\alpha.$  In this case \[\varkappa<\frac{1}{3}\min\left(\frac{\alpha(1-2\alpha)}{1-\alpha},\varkappa_1\right)=\frac{\alpha(1-2\alpha)}{3(1-\alpha)}.\]

\begin{figure}[h]
\centering
\includegraphics[width=0.6\linewidth,height=0.4\linewidth]{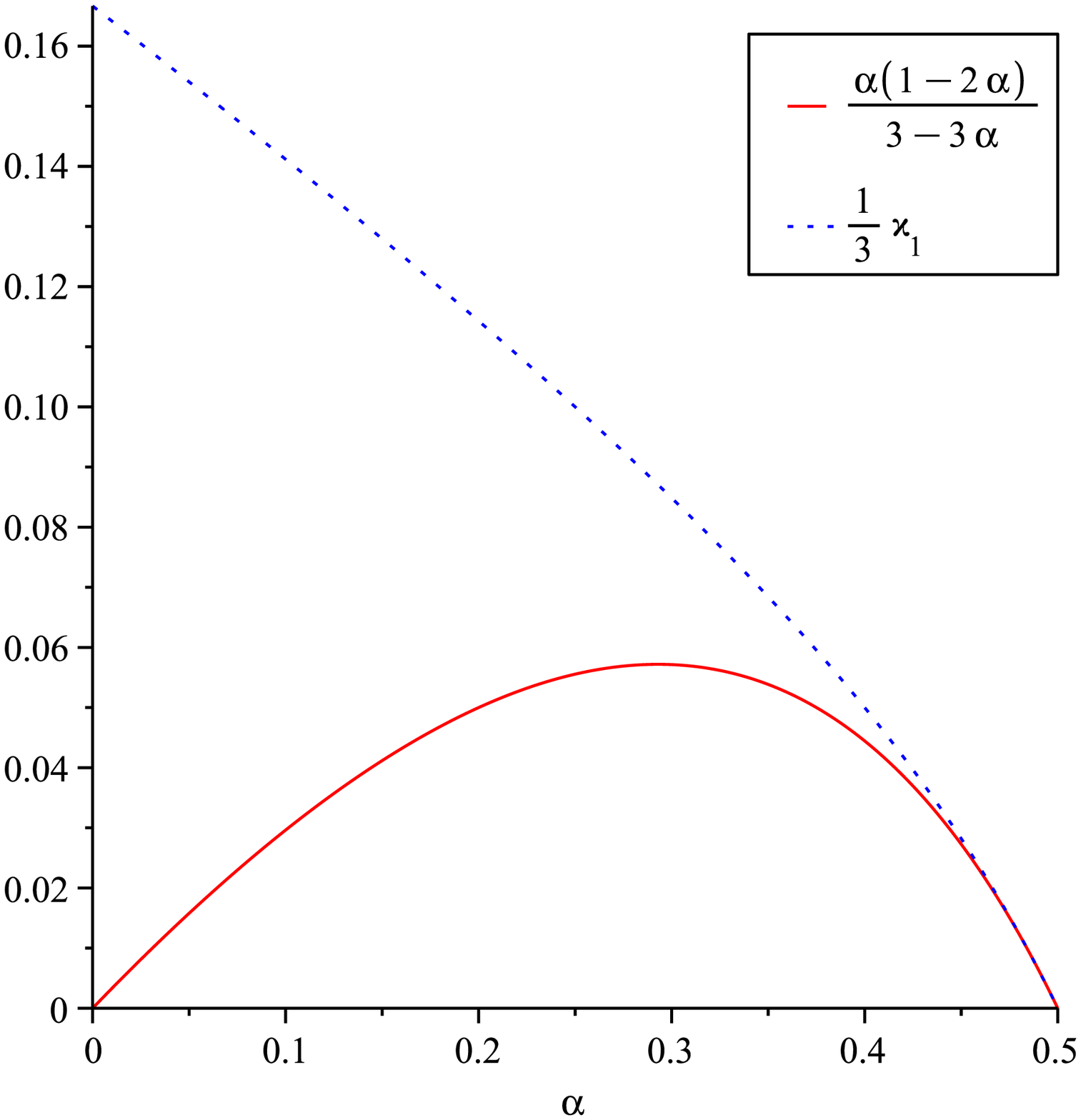}
\caption{Graphs of $\varkappa_1/3$ and $\frac{\alpha(1-2\alpha)}{3(1-\alpha)}$ for $d=1.$}
\label{fig1}
\end{figure}
\end{example}

\begin{example}\label{ex3}
Let us consider  $\eta (x),$ $x\in\mathbb{R}^d,$ with the covariance function of the form
\[
B\left( \left\Vert x\right\Vert\right) =\left( 1+\left\Vert x\right\Vert^{2}\right)
^{-\theta},\quad  \theta >0.
\]
In geostatistics, it is known as the Cauchy covariance function \cite{gne,lim1, wac}.

The corresponding spectral density  has the form, see \cite[Proposition 2.4]{lim1},
\begin{equation}\label{den1}
f\left( \left\| \lambda \right\|\right) =\frac{\left\| \lambda \right\|^{\theta-\frac{d}{2}}}{2^{\frac{d}{2}+\theta-1}\pi^\frac{d}{2} \Gamma(\theta)}\,
K_{\frac{d}{2}-\theta}(\left\| \lambda \right\|),
\end{equation}
where $K_\mu(\cdot)$ is the modified Bessel function of the second kind.

It follows from the representation
$
B\left( \left\Vert x\right\Vert\right) =\left\Vert x\right\Vert^{-2\theta}\left( 1+\left\Vert x\right\Vert^{-2}\right)
^{-\theta}
$
that, in the notations of the paper, $\alpha=2\theta,$   $L(t)=\left( 1+t^{-2}\right)^{-\theta},$ and the  Cauchy covariance function satisfies Assumption~\ref{ass1}. The considered case of long-range dependence corresponds to the region $0<\theta<d/4.$   By Remark~\ref{rem5} the slowly varying function $L(\cdot)$ satisfies the condition~(\ref{gr}) for $q<\min(2,{d}/{2}-\alpha).$

To verify the condition~(\ref{f}) we note that
\[
f(\left\| \lambda \right\| )=  c_2(d,2\theta)\left\| \lambda \right\| ^{2\theta
-d}\left( 1+\left\| \lambda \right\|^{2}\right)^{-\theta}+\varepsilon(\left\| \lambda \right\|),\]
where
\begin{equation}\label{eps}
\varepsilon(\left\| \lambda \right\|)=f(\left\| \lambda \right\| )-c_2(d,2\theta )\left\| \lambda \right\| ^{2\theta
-d}\left( 1+\left\| \lambda \right\|^{2}\right)^{-\theta}.
\end{equation}

We will use the asymptotic expansions of the functions $\left( 1+\left\| \lambda \right\|^{2}\right)^{-\theta}$  and $f(\left\| \lambda \right\| )$ for
$\left\| \lambda \right\|\to 0.$ 
By the binomial series expansion we get

\begin{equation}\label{bin}\left( 1+\left\| \lambda \right\|^{2}\right)^{-\theta}=\sum_{j=0}^\infty\binom{-\theta}{j}\left\| \lambda \right\|^{ 2j}=1-\theta\left\| \lambda \right\|^{2}+\frac{\theta(\theta+1)}{2}\left\| \lambda \right\|^{4}+...\,.\end{equation}

Note that the condition of long-range dependence $\theta\in (0,d/4)$ implies positivity of $d/2-\theta.$ First we consider the case  $d/2-\theta\not\in\mathbb{N}.$ 
By the asymptotic expansions of $f(\left\| \lambda \right\| )$ given in \cite[Proposition 3.2]{lim1} and Euler's reflection formula
$\Gamma(z)\Gamma(1-z) = {\pi}/{\sin{(\pi z)}}$   we obtain

\[f(\left\| \lambda \right\| )=\frac{1}{2^{d}\pi^\frac{d-2}{2} \Gamma(\theta)\sin\left(\pi\left(\frac{d}{2}-\theta\right)\right)}\sum_{j=0}^{\infty}\left(\frac{(\left\| \lambda \right\|/2)^{2j+2\theta-d}}{j!\,\Gamma\left(j+\theta-\frac{d-2}{2}\right)}-
\frac{(\left\| \lambda \right\|/2)^{2j}}{j!\,\Gamma\left(j+\frac{d+2}{2}-\theta\right)}\right)\]
\[=\frac{1}{2^{d}\pi^\frac{d-2}{2} \Gamma(\theta)\sin\left(\pi\left(\frac{d}{2}-\theta\right)\right)}\left(\frac{(\left\| \lambda \right\|/2)^{2\theta-d}}{\Gamma\left(\theta-\frac{d-2}{2}\right)}+\frac{(\left\| \lambda \right\|/2)^{2+2\theta-d}}{\Gamma\left(1+\theta-\frac{d-2}{2}\right)}+...\right.\]
\[\left.-\frac{1}{\Gamma\left(\frac{d+2}{2}-\theta\right)}-...\right)=c_2(d,2\theta)\left\| \lambda \right\|^{2\theta-d}\left(1+\frac{\left\| \lambda \right\|^2}{2\left(2+2\theta-d\right)}+...\right.\]
\begin{equation}\label{f2}\left.-\frac{2^{2\theta-d}\Gamma\left(\theta-\frac{d-2}{2}\right)}{\Gamma\left(\frac{d+2}{2}-\theta\right)}\left\| \lambda \right\|^{d-2\theta}-...\right).\end{equation}

Therefore, by the substitution of (\ref{bin}) and (\ref{f2}) into (\ref{eps})  for $d/2-\theta\not\in\mathbb{N}$  we obtain  $\varepsilon(t)=t^{2\theta
-d}L\left( 1/t\right)\cdot{\cal O}\left(t^\upsilon\right),$ $t\to 0,$ where
\begin{equation}\label{up0}\upsilon= \min_{\theta<d/4}(2,d-2\theta)=
\begin{cases}d-2\theta,& \text{if}\ d=1, 2,\ \text{or}\ d=3\ \text{and}\ \theta\in (1/2,3/4);\\
2,& \text{if}\ d\ge 4,\  \text{or}\ d=3\ \text{and}\ \theta\in(0,1/2).
\end{cases}
\end{equation}

Now we investigate the case $l:=d/2-\theta\in\mathbb{N}.$
Proposition 3.2 \cite{lim1} implies 

\[f(\left\| \lambda \right\| )=\frac{1}{2^{d}\pi^\frac{d}{2} \Gamma(\frac{d}{2}-l)}\left(\sum_{j=0}^{l-1}(-1)^j\frac{(l-j-1)!(\left\| \lambda \right\|/2)^{2j-2l}}{j!}+
(-1)^{l+1}\right.\]
\[\times\sum_{j=0}^{\infty}\frac{(\left\| \lambda \right\|/2)^{2j}}{j!(l+j)!}\left(2\ln(\left\| \lambda \right\|/2)-\psi(j+1)-\psi(l+j+1)\right)
\Bigg)=\frac{1}{2^{d}\pi^\frac{d}{2} \Gamma(\frac{d}{2}-l)}\]
\[\times\left(\frac{2^{2l}(l-1)!}{\left\| \lambda \right\|^{2l}}-\frac{2^{2l-2}(l-2)!}{\left\| \lambda \right\|^{2l-2}}+...\right.\left.+(-1)^{l+1}\frac{2\ln(\left\| \lambda \right\|/2)-\psi(1)-\psi(l+1)}{l!}+...\right)\]
\[=c_2(d,2\theta)\left\| \lambda \right\|^{2\theta-d}\left(1+\frac{\left\| \lambda \right\|^2}{2\left(2+2\theta-d\right)}+...+\frac{(-1)^{d/2+1-\theta}\,\ln(\left\| \lambda \right\|)\,\left\| \lambda \right\|^{d-2\theta}}{2^{d-2\theta-1}(\frac{d}{2}-1-\theta)!(\frac{d}{2}-\theta)!}+...\right),\]
where $\psi(z):={\Gamma'(z)}/{\Gamma(z)}$ is the digamma function.

Hence, for  $d/2-\theta\in\mathbb{N}$ we get  
\begin{equation}\label{up1}\upsilon= 
\begin{cases}2-\delta,&\text{if}\ \ \theta=1/2,\, d=3;\\
2,& \text{otherwise},
\end{cases}
\end{equation}
where $\delta$ is an arbitrary non-negative number.

Now we consider the case $\left\| \lambda \right\| \to +\infty.$ By (\ref{den1}) and  the asymptotic property $K_\mu(t)={\cal O}\left(e^{-t}\right),$ $t\to +\infty,$ see \cite[(8.451.6)]{gra}, we obtain $f(\left\| \lambda \right\|)={\cal O}\left(\left\| \lambda \right\|^{-d}\right),$ when $\left\| \lambda \right\| \to +\infty.$ Thus, by Remark~\ref{rem6}  Assumption~{\rm\ref{ass2}} holds true for the Cauchy model.  

If ${d}/{2}-\alpha<2,$ then $q$ can be chosen close to ${d}/{2}-\alpha$ and  (\ref{kappa_1}) holds. For $q=2$ it follows from (\ref{up0}) and (\ref{up1}) that the inequality $2\upsilon(1-\gamma)-2\beta<2q-2\beta$ holds true for all $\upsilon.$ Hence, by the definition of $\varkappa_0,$ we can choose  $\varkappa_1$ as in (\ref{kappa_1}).
\end{example}
\begin{example}
Let us consider  $\eta (x),$ $x\in\mathbb{R}^d,$ with the covariance function of the form
\[
B\left( \left\Vert x\right\Vert\right) =\left( 1+\left\Vert x\right\Vert^{{}\sigma}\right)
^{-\theta},\quad \sigma \in \left( 0,2\right] ,\ \theta >0,
\]
which is known as the generalized Linnik covariance function \cite{erg,kot,lim1}. Cauchy and Linnik's fields are important  particular cases of this model.

The Cauchy model with $\sigma=2$ was considered in Example~\ref{ex3}. Therefore, we will only investigate  the case $\sigma \in (0,2)$ and $\theta >0.$ Note, that the asymptotic expansion of $f\left( \left\| \lambda \right\|\right)$ for $\sigma \in (0,2)$ differs from the expansion for $\sigma =2.$ That is why we consider these two cases of the generalized Linnik model separately.  

For $\sigma \in (0,2)$ the spectral density  has the form, see \cite[Proposition 2.4]{lim1}, 
\[
f\left( \left\| \lambda \right\|\right) =-\frac{\left\| \lambda \right\|^\frac{2-d}{2}}{2^\frac{d-2}{2}\pi^\frac{d+2}{2} }\operatorname{Im}
\int_{0}^{\infty }\frac{K_\frac{d-2}{2}(\left\| \lambda \right\|u)\,u^\frac{d}{2}\,\mathrm{d}u}{\left( 1+e^{i\pi \sigma
/2}u^\sigma\right)^{\theta}}.
\]

Analogously to the case of the Cauchy field,  $\alpha=\theta\sigma,$ $\theta\sigma<d/2,$  $L(t)=\left( 1+t^{-\sigma}\right)^{-\theta},$ and the generalized Linnik covariance function satisfies Assumption~\ref{ass1}.   By Remark~\ref{rem5} the condition~(\ref{gr}) holds true with $q<\min(\sigma,{d}/{2}-\alpha).$

Note that
\[
f(\left\| \lambda \right\| )=  c_2(d,\theta\sigma )\left\| \lambda \right\| ^{\theta\sigma
-d}\left( 1+\left\| \lambda \right\|^{\sigma}\right)^{-\theta}+\varepsilon(\left\| \lambda \right\|),\]
\[\varepsilon(\left\| \lambda \right\|)=f(\left\| \lambda \right\| )-c_2(d,\theta\sigma )\left\| \lambda \right\| ^{\theta\sigma
-d}\left( 1+\left\| \lambda \right\|^{\sigma}\right)^{-\theta},\]
\[\left( 1+\left\| \lambda \right\|^{\sigma}\right)^{-\theta}=1-\theta\left\| \lambda \right\|^{\sigma}+\frac{\theta(\theta+1)}{2}\left\| \lambda \right\|^{2 \sigma}+...\,, \quad \left\| \lambda \right\|\to 0.\]

By the asymptotic expansions of $f(\left\| \lambda \right\| ),$ see \cite[Proposition 3.9]{lim1}, we obtain
\[f(\left\| \lambda \right\| )=\frac{1}{2^{d}\pi^\frac{d}{2} \Gamma(\theta)}\left(\frac{\Gamma\left(\theta\right)\Gamma\left(\frac{d-\sigma\theta}{2}\right)}{\Gamma\left(\frac{\sigma\theta}{2}\right)} \left(\frac{\left\| \lambda \right\|}{2}\right)^{\sigma\theta-d}-(1-\chi_{\raisebox{-3pt}{\scriptsize $d+2\mathbb{N}_0$}}(\sigma\theta+\sigma))
\right.\]
\[\times \frac{\Gamma\left(\theta+1\right)\Gamma\left(\frac{d-\sigma\theta-\sigma}{2}\right)}{\Gamma\left(\frac{\sigma(\theta+1)}{2}\right)} \left(\frac{\left\| \lambda \right\|}{2}\right)^{\sigma\theta-d+\sigma}+...+(1-\chi_{\raisebox{-3pt}{\scriptsize $\mathbb{N}_0$}}\left({d}/{\sigma}-\theta\right))
\frac{2\Gamma\left(\frac{d}{\sigma}\right)\Gamma\left(\frac{\sigma\theta-d}{\sigma}\right)}{\sigma\Gamma\left(\frac{d}{2}\right)}+...
\]
\[+\left.\frac{(-1)^{\frac{\sigma\theta+\sigma-d}{2}}\Gamma\left(\theta+1\right)}{\left(\frac{\sigma\theta+\sigma-d}{2}\right)!\cdot\Gamma\left(\frac{\sigma(\theta+1)}{2}\right)} \chi_{\raisebox{-3pt}{\scriptsize $d+2\mathbb{N}_0$}}(\sigma\theta+\sigma) \ln\left(\left\| \lambda \right\|\right) \left(\frac{\left\| \lambda \right\|}{2}\right)^{\sigma\theta+\sigma-d}+...\right)\,,\]
where $d+2\mathbb{N}_0:=\{m:\, m=d+2j,\, j\in \mathbb{N}_0\}.$

Therefore,   $\varepsilon(t)=t^{2\theta
-d}L\left( 1/t\right)\cdot{\cal O}\left(t^\upsilon\right),$ $t\to 0,$ when  $\sigma\theta<d/2$ and
\begin{equation}\label{up2}\upsilon=
\begin{cases}d-\sigma\theta,&\text{if}\   \sigma\in(d/(\theta+1),2),\  \theta\in(0,1),\ d=1,2,3;\\
\sigma-\delta,&\text{if}\  \sigma=d/(\theta+1),\  \theta\in(\max(0,d/2-1),1),\ d=1,2,3;\\
\sigma,& \text{if}\ d\ge 4,\ \text{or}\ \sigma\in(0,d/(\theta+1))\  \text{and}\ d=1,2,3,
\end{cases}
\end{equation}
where $\delta$ is an arbitrary non-negative number.

Now we consider the case $\left\| \lambda \right\| \to +\infty.$ By \cite[Proposition 3.4]{lim1} we get $f(\left\| \lambda \right\| )={\cal O}\left( \left\| \lambda \right\|^{-d-\sigma}\right),$ $\left\| \lambda \right\|\to +\infty.$ Hence, by Remark~\ref{rem6}  Assumption~{\rm\ref{ass2}} holds true for the generalized Linnik model.

If ${d}/{2}-\alpha<\sigma,$ then $q$ can be chosen close to ${d}/{2}-\alpha$ and  (\ref{kappa_1}) holds. For $q=\sigma$  by (\ref{up2})  we obtain that the inequality $2\upsilon(1-\gamma)-2\beta<2q-2\beta$ holds for all $\upsilon.$ Hence, it follows from the definition of $\varkappa_0$ that  $\varkappa_1$ can be given by (\ref{kappa_1}).

Figure~\ref{fig2} displays  the graphs of $\frac{\alpha(d-2\alpha)}{3(d-\alpha)}$ and $\varkappa_1/3$ for $d=2,$ $q=\sigma=7/4,$  $\upsilon=17/12,$ $\theta=1/3,$ and 
 $d=3,$ $q=\upsilon=\sigma=1,$ plotted as  functions of the variable $\alpha=\theta.$ Notice, that contrary to Example~\ref{ex6} and the case $d=2$ the order $\varkappa_1/3$ is smaller than  $\frac{\alpha(d-2\alpha)}{3(d-\alpha)}$ for $d=3$ and the values of $\alpha$ close to $d/2.$
\begin{figure}[h]
\begin{minipage}{8cm}
\includegraphics[width=1\linewidth,height=0.8\linewidth]{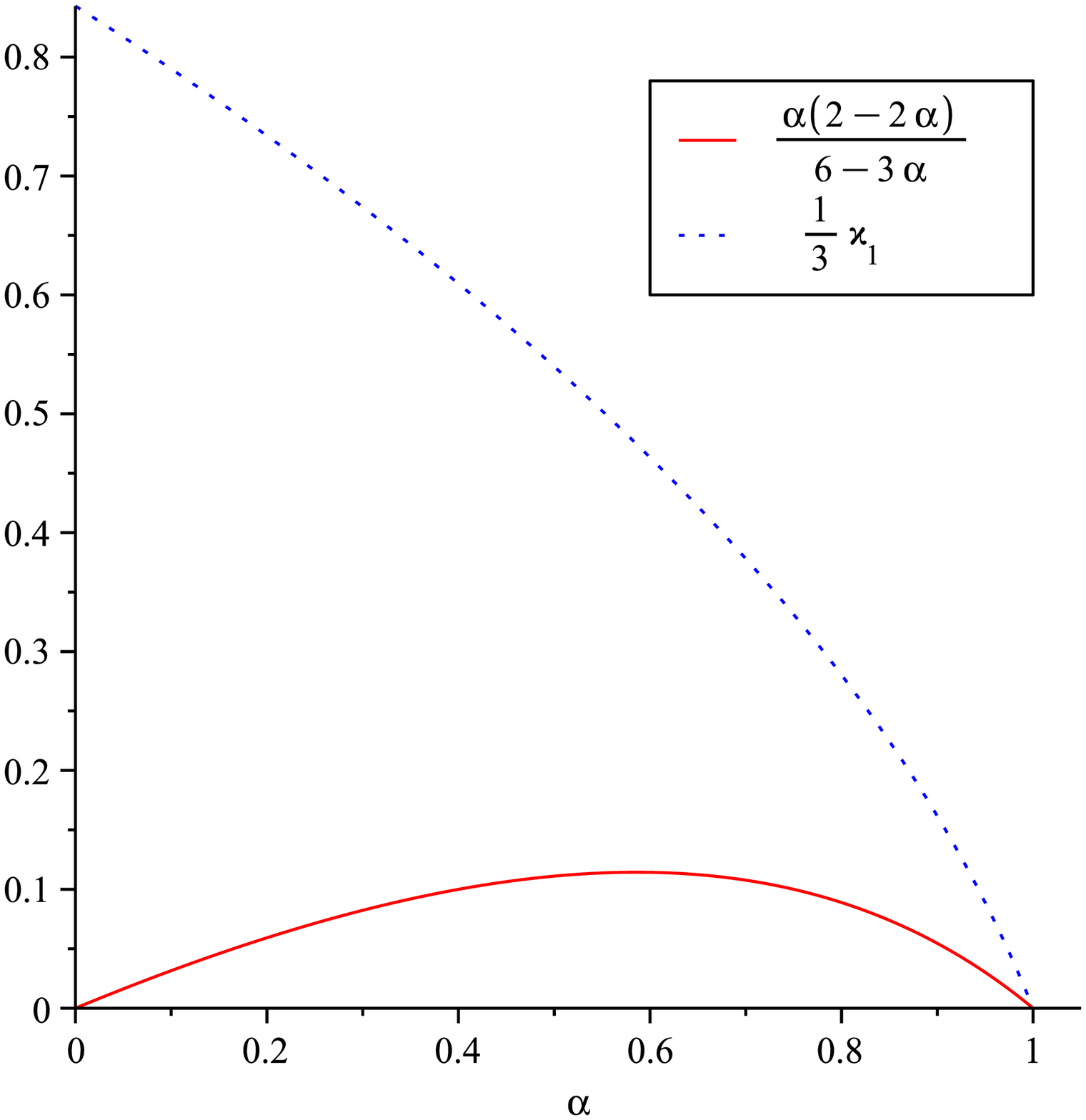}
 \end{minipage} \
\begin{minipage}{8cm}
\includegraphics[width=1\linewidth,height=0.8\linewidth]{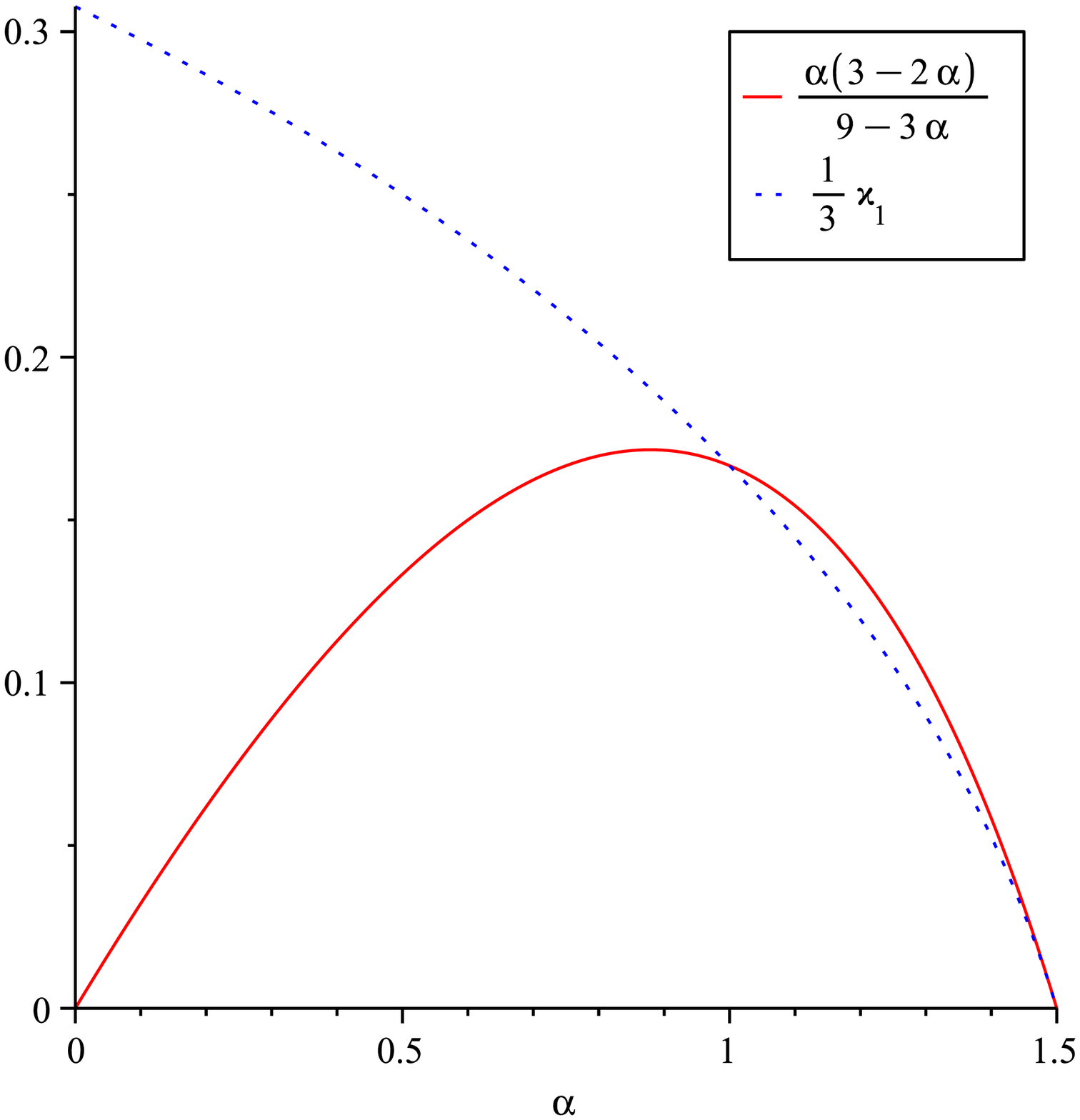}
\end{minipage}
\caption{Graphs of $\varkappa_1/3$ and $\frac{\alpha(d-2\alpha)}{3(d
-\alpha)}$ for $d=2$  and 3.}
\label{fig2}
\end{figure}
\end{example}

\begin{rem}
Due to the strict inequality for $\varkappa$ in Theorem~\ref{rate} the constant $\delta,$ appearing in expressions (\ref{up1}) and (\ref{up2}) for $\upsilon,$ can be chosen equal to zero.
\end{rem}

\section{Concluding Remarks}\label{sec6}
We have investigated the rate of convergence to the Rosenblatt-type limit distributions in the non-central limit theorem. The results were obtained under rather general assumptions allowing specifications for various scenarios.  In particular, special attention was devoted to the Cauchy, generalized Linnik, and local-global distinguisher random processes and fields.  
We use direct analytical probabilistic methods which have, in our view, an independent interest as an alternative to methods in \cite{bre,nor,nor1}.
The analysis and the approach to the rate of convergence in non-central limit theorems are new and extend the investigations of the rate of convergence in central limit theorems in the current literature.

In the one-dimensional case the rate of convergence of $K_{r,2}$ obtained in the proof of Theorem~\ref{rate} is analogous to  the result for the discrete fractional Gaussian noise in \cite{bre}. However, Theorem~\ref{rate} additionally estimates the rate of the term $S_r,$ which allows to consider the class of all functions of Hermite rank 2.  Moreover, the obtained results are valid for the multidimensional case and more general classes of covariance functions and random processes. 

It is possible to extend the results to wider classes of slowly varying functions with remainder, see \cite[\S 3.12]{bin}, whose bounds are different from (\ref{gr}) (the detailed discussion on the condition (\ref{gr}) is given in Remarks~\ref{rem5} and ~\ref{rem6}). However, for such classes the rate of convergence would be different (depending on the remainder) from the following results.  Assumption~\ref{ass2} was chosen to ensure a polynomial convergence rate.

\section{Acknowledgements }
This work was partly supported by La Trobe University DRP Grant in Mathematical and Computing
Sciences. The authors are also grateful for the referee's careful reading of the paper and many detailed comments and suggestions, which helped to improve the paper.

\end{document}